\documentclass{amsart}
\usepackage{amsmath, amssymb,epic,graphicx,mathrsfs,enumerate}
\usepackage[all]{xy}

\usepackage{amsthm}
\usepackage{amssymb}
\usepackage{latexsym}
\usepackage{longtable}
\usepackage{epsfig}
\usepackage{amsmath}
\usepackage{hhline}

\newtheorem{thm}{Theorem}[section]

\newtheorem{lem}[thm]{Lemma}

\begin{document}

\title{On The Non-coprime $k(GV)$ Problem}
\author{Robert M. Guralnick}
\address{Department of Mathematics, University of Southern California,
Los Angeles, CA 90089-2532, USA} \email{guralnic@usc.edu}
\author{Attila Mar\'oti}
\address{Alfr\'ed R\'enyi Institute of Mathematics, Re\'altanoda utca 13-15, H-1053, Budapest, Hungary}
\email{maroti.attila@renyi.mta.hu} \keywords{linear
representation, number of conjugacy classes}
\subjclass[2000]{20C99}
\thanks{The first author was partially supported by NSF
grant DMS-1001962. The research of the second author was supported
by a Marie Curie International Reintegration Grant within the 7th
European Community Framework Programme, by the J\'anos Bolyai
Research Scholarship of the Hungarian Academy of Sciences, and by
OTKA K84233.}
\date{\today}

\begin{abstract}
Let $V$ be a finite faithful completely reducible $FG$-module for
a finite field $F$ and a finite group $G$. In various cases
explicit linear bounds in $|V|$ are given for the numbers of
conjugacy classes $k(GV)$ and $k(G)$ of the semidirect product
$GV$ and of the group $G$ respectively. These results concern the
so-called non-coprime $k(GV)$-problem.
\end{abstract}
\maketitle

\section{Introduction}

The topic of this paper originates from the long-standing
$k(B)$-conjecture of Brauer which states that the number $k(B)$ of
complex irreducible characters in any $p$-block $B$ of any finite
group $G$ is at most the order of the defect group of $B$. Nagao
\cite{N} showed that for $p$-solvable groups $G$ Brauer's
$k(B)$-problem is equivalent to the so-called $k(GV)$-problem
which is described in the next paragraph.

For a finite group $X$ let $k(X)$ be the number of conjugacy
classes of $X$. Let $V$ be a finite faithful $FG$-module for some
finite field $F$ of characteristic $p$ and finite group $G$. Form
the semidirect product $GV$. The $k(GV)$-problem states that
$k(GV) \leq |V|$ whenever $(|G|,|V|)=1$. Works of Kn\"orr, Gow,
and especially Robinson, Thompson \cite{RT} have led to
fundamental breakthroughs in attacking the $k(GV)$-problem which
have culminated in a complete solution of the problem, with the
final step completed by Gluck, Magaard, Riese, Schmid \cite{GMRS}.
The full solution of the problem (not counting the Classification
of Finite Simple Groups) is written in the book \cite{Schmid}.

Let $V$, $G$, and $F$ be as in the previous paragraph with not
assuming $(|G|,|F|)=1$ but that $V$ is completely reducible. Is
there a universal constant $c$ such that $k(GV) \leq c^{n}|V|$
where $n$ denotes the $F$-dimension of $V$? Can $c$ be taken to be
$1$ in most cases? This is a weak version of the so-called
non-coprime $k(GV)$-problem \cite[Problem 1.1]{GT} which is also
important in a character theoretic point of view. Indeed, as
pointed out by Robinson, the inequality $k(GV) \leq |V|$ combined
with \cite[Lemma 5]{Rob} would imply the $k(B)$-conjecture for a
wider class of $p$-constrained groups.

There has been progress made on the non-coprime $k(GV)$-problem.
Kov\'acs and Robinson \cite[Theorem 4.1]{KR} gave an affirmative
answer to our first question above in case $G$ is a $p$-solvable
group. In fact, for $p$-solvable groups, Liebeck and Pyber
\cite{LP} showed that $c$ can be taken to be $103$. Guralnick and
Tiep \cite{GT} have proved $k(GV) < |V|/2$ for many almost
quasi-simple groups $G$, Keller \cite{K1} has obtained results in
case $V$ is an imprimitive irreducible module, and there are some
interesting character theoretic arguments developed by Keller
\cite{K2}.

Our first result concerns the extraspecial case of the non-coprime
$k(GV)$-problem.

\begin{thm}
\label{extraspecial} Let $r$ be a prime and let $R$ be an
$r$-group of symplectic type with $|R/Z(R)| = r^{2a}$ for some
positive integer $a$. Let $V$ be a faithful, absolutely
irreducible $KR$-module of dimension $r^{a}$ for some finite field
$K$. View $V$ as an $F$-vector space where $F$ is the prime field
of $K$. Let $G$ be a subgroup of $GL(V)$ which contains $R$ as a
normal subgroup. Then $k(GV) \leq |V|$ unless one of the following
cases holds.
\begin{enumerate}
\item[(1)] $r^{a} = 2^{6}$ and $|K| = 3$. In this case $k(GV) \leq
2^{120}$.

\item[(2)] $r^{a} = 2^{5}$ and $|K| = 3$, $5$, $7$, $9$, or $11$.
In this case $k(GV) \leq 2^{119}$.

\item[(3)] $r^{a} = 3^{3}$ and $|K| = 4$ or $7$. In this case
$k(GV) \leq 2^{82}$.

\item[(4)] $r^{a} = 2^{4}$ and $|K| = 3$, $5$, $7$, $9$, $17$,
$25$, or $27$. In this case $k(GV) \leq 2^{82}$.

\item[(5)] $r^{a} = 2^{3}$ and $|K| = 3$, $5$, $7$, $9$, $25$,
$27$, $49$, $81$, or $125$. We have $k(GV) \leq 2^{58}$.

\item[(6)] $r^{a} = 3^{2}$ and $|K| = 4$, $16$, or $25$. In this
case $k(GV) \leq 2^{44}$.

\item[(7)] $r^{a} = 2^{2}$ and $|K| = 3$, $5$, $9$, $25$, $27$,
$81$, $125$, or $243$. We have $k(GV) \leq 2^{32}$.
\end{enumerate}
\end{thm}

The next result deals with the case where $G$ is a meta-cyclic
group. Here the bound $|V|$ is best possible in infinitely many
cases. (Just consider a Singer cycle $G$ acting on $V$.)

\begin{thm}
\label{main} Let $V$ be an $n$-dimensional finite vector space
over the field of $p$ elements where $p$ is a prime. The group $X
= GL(1,p^{n}).n$ acts naturally on $V$. Then for any subgroup $G$
of $X$ we have $k(GV) \leq |V|$ unless $GV \cong D_{8}$ or $S_4$.
\end{thm}

The ideas in the proof of Theorem \ref{extraspecial} together with
Theorem \ref{main} yield a general result on $k(GV)$ in case the
group $G$ has nilpotent generalized Fitting subgroup and when $V$
is a faithful primitive irreducible module.

\begin{thm}
\label{primitive} Let $V$ be a finite faithful primitive
irreducible $FG$-module for some group $G$ with
$\mathrm{Fit}^{*}(G) = \mathrm{Fit}(G)$. Then $k(GV) \leq \max \{
|V|, 2^{1344} \}$.
\end{thm}

What can be said about $k(G)$ in the setting of the non-coprime
$k(GV)$-problem? Clearly, $k(G) \leq k(GV)$. Interestingly, in
case $(|G|,|V|)=1$, the fact that $k(G) \leq |V|$ was only derived
from the full solution of the $k(GV)$-problem. Is it true that
$k(G) \leq |V|$ whenever $V$ is a completely reducible module? We
make a first step in answering this question.

\begin{thm}
\label{irreducible} Let $V$ be a finite faithful irreducible
$FG$-module for some finite field $F$ and finite group $G$.
Suppose that $V$ can be induced from a primitive irreducible
$FL$-module $W$ for some finite group $L$ with $k(N) <
|W|/\sqrt{3}$ for every normal subgroup $N$ of $L/C_{L}(W)$. Then
$k(G) < (2/3)|V|$.
\end{thm}

Note that the bound $k(N) < |W|/\sqrt{3}$ in Theorem
\ref{irreducible} is satisfied for all normal subgroups $N$ of
`many' $L$. For example, if $L$ is the $G$ and $W$ in the $V$
considered in Section 5, then $|L| < |W|/\sqrt{3}$ for
sufficiently large characteristics.

Finally, it may be possible that the $2$-power estimates for
$k(GV)$ in Theorem \ref{extraspecial} and in Theorem
\ref{primitive} can all be taken to be $11$ as $k(AGL(2,3)) = 11$.

\section{Bounding the dimensions of eigenspaces}

Throughout this section we will use the following notations and
assumptions. Let $r$ be a prime. An $r$-group $R$ is said to be of
symplectic type if either $r$ is odd and $R$ is extraspecial of
exponent $r$, or $r=2$, $R/Z(R)$ is elementary abelian, $R'$ has
order $2$, $R$ has exponent $4$ and $Z(R)$ has order $2$ (in which
case $R$ is extraspecial) or has order $4$. Let $V$ denote a
faithful irreducible $FG$-module where $F$ is an algebraically
closed field of characteristic $p > 0$ and $G$ is a finite group.
Suppose that the group $G$ has a normal subgroup $R$ of symplectic
type with $|R/Z(R)| = r^{2a}$ for some prime $r$ and positive
integer $a$. This $r$-group $R$ acts absolutely irreducibly on
$V$, $\dim_{F}(V) = r^{a}$, and $O_{p}(G) = 1$. Suppose that $R$
is such that $Z(G) = Z(R)$. The non-identity elements of $G/R$ act
faithfully on $R/Z(R)$ and trivially on $Z(R)$.

Let $x$ be an element of $G$. For a field element $\lambda \in F$
we denote the eigenspace of $\lambda$ of a matrix representation
of $x$ on $V$ by $\mathrm{Eigen}(\lambda,x)$. In this section we
wish to bound $d(x) = \max_{\lambda \in F} \dim
(\mathrm{Eigen}(\lambda,x))$ (but in the end we will be interested
in $\dim_{F}(C_{V}(x))$).

Let the element $x$ be in $R$. If $x$ is central then $d(x) =
r^{a}$. Otherwise if $x$ is non-central in $R$ then the value of
the character (of $V$) at $x$ is $0$ and so $d(x) = r^{a-1}$.

From now on, assume that $x \in G \setminus R$.

The following important theorem considers the case when $\langle x
\rangle$ is irreducible on the vector space $R/Z(R)$ and has order
a power of $p$.

\begin{thm}[Hall-Higman, \cite{HH}]
Use the notations and assumptions of this section. Let $x$ be an
element of $G \setminus R$ of prime power order $q$ divisible by
$p$, the characteristic of the field $F$. Assume that $\langle x
\rangle$ acts irreducibly on $R/Z(R)$ (where $|R/Z(R)|=r^{2a}$).
Then there exists a non-negative integer $b$ so that $\dim(V) =
r^{a} = (q-1) + bq$, and the Jordan canonical form of $x$ on $V$
consists of $b+1$ Jordan blocks, $b$ of size $q$ and $1$ of size
$q-1$. In particular, $d(x) = \dim(C_{V}(x)) = b+1 = (r^{a}+1)/q$.
\end{thm}

It is necessary to say a few words about the proof of the
Hall-Higman theorem. Put $x$ (viewed as a linear transformation of
$V$) in Jordan canonical form. Suppose that $x$ has $m$ Jordan
blocks of sizes: $a_{1}, \ldots , a_{m}$. We seek to find the
$a_i$'s explicitly. We certainly have one restriction, namely,
$\dim(V) = r^{a} = \sum_{i=1}^{m} a_{i}$. For another one, let $E$
be the enveloping algebra of the group of linear transformations
$R$ of $V$. Then $E = \mathrm{End}(V)$ (and so $\dim_{F}(E) =
r^{2a}$). Hall and Higman proceed to calculate
$\dim_{F}(C_{E}(x))$ in two different ways. On one hand, this is
$\sum_{i=1}^{m} (2i-1)a_{i}$, while on the other, it is $1+
(r^{2a}-1)/q$, the number of $\langle x \rangle$-orbits of the set
$R/Z(R)$. This gives our second restriction on the $a_{i}$'s. It
turns out that these two restrictions are sufficient to determine
the $m$ non-negative integers.

Now let $x$ be an element of $G \setminus R$ of prime power order
$q$ such that $p$ does {\it not} divide $q$. As before, suppose
that $\langle x \rangle$ is irreducible on the vector space
$R/Z(R)$. By this we are also assuming that $q$ is not a power of
$r$. (The following argument is taken from the series of exercises
in \cite[Pages 371-372]{G}.) In this case the Jordan canonical
form of $x$ on $V$ is a diagonal matrix (since $F$ is
algebraically closed). Let the number of distinct eigenvalues of
$x$ be $m$, and let $a_{i}$ be the multiplicity of the $i$-th
eigenvalue. Then $r^{a} = \sum_{i=1}^{m} a_{i}$. Again, let $E$ be
the enveloping algebra of $R$. Clearly, $\dim_{F}(C_{E}(x)) =
\sum_{i=1}^{m} {a_{i}}^{2}$. On the other hand, it is easy to see
that $\dim_{F}(C_{E}(x))$ is again the number of $\langle x
\rangle$-orbits of the set $R/Z(R)$. This gives us two equations
involving the $a_{i}$'s which are sufficient to determine the $m$
non-negative integers we are looking for. In particular, we find
that the multiplicity of any eigenvalue is at most $(r^{a}+1)/q$.
Hence $d(x) \leq (r^{a}+1)/q$ (as in the case when $q$ {\it was} a
power of $p$).

Using the same argument as before, one can show even more.

\begin{lem}
\label{HH} Use the notations and assumptions of this section. Let
$x$ be an element of $G \setminus R$ so that $\langle x \rangle$
is irreducible on $R/Z(R)$. Let the order of $x$ be $m$. (The
positive integer $m$ divides $r^{2a}-1$.) Then $d(x) \leq
(r^{a}+1)/m$.
\end{lem}

The Jordan canonical form of a matrix is a block matrix consisting
of Jordan blocks in the main diagonal and zero matrices everywhere
else where a Jordan block is a block matrix with the same
companion matrix in the diagonal, identity matrices just above the
diagonal and zero matrices everywhere else.

At this point let us mention another result.

\begin{lem}
\label{HHtwoblocks} Use the notations and assumptions of this
section. Let $x$ be an element of $G \setminus R$, and let $R_1$,
$R_2$ be two maximal abelian subgroups of $R$ whose intersection
is $Z(R)$. Suppose that the Jordan canonical form of $x$ on
$R/Z(R)$ consists of two $a$-by-$a$ Jordan blocks that are the
same where one leaves $R_{1}/Z(R)$ invariant and the other leaves
$R_{2}/Z(R)$ invariant. Suppose that $\langle x \rangle$ is
irreducible on both $R_{1}/Z(R)$ and on $R_{2}/Z(R)$, and $x$ has
order $m$. Then $d(x) \leq 1 + (r^{a}-1)/m$.
\end{lem}

\begin{proof}
The group $R$, which is a product of the maximal abelian subgroups
$R_1$ and $R_2$, acts absolutely irreducibly on the vector space
$V$. We can diagonalize $R_1$ (and $R_{2}$) on $V$ and all
eigenspaces are one dimensional.

Apart from a single eigenspace, the element $x$ permutes all other
eigenspaces in regular orbits. This means that $V$ is a direct sum
of a single module of dimension $1$ and some free $\langle x
\rangle$-modules. Hence, $d(x) \leq 1 + (r^{a}-1)/m$.
\end{proof}

Let us modify the proof of the Hall-Higman theorem to include the
case when $x$ is an element of order a power of $r$. In this case
we cannot assume irreducibility. Instead, suppose that the Jordan
canonical form of our element $x \in G \setminus R$ viewed as a
linear transformation of $R/Z(R)$ consists of a unique Jordan
block of size $2a$ or of two Jordan blocks each of size $a$. In
the latter case suppose that $x$ leaves two maximal totally
singular subspaces of $R/Z(R)$ invariant, both of order $r^{a}$.

Notice that a Jordan block of an $r$-element is a matrix with
$1$'s in the main diagonal, $1$'s in the diagonal just above the
main diagonal, and $0$'s elsewhere. Hence the order of $x$ is
$r^{k}$ and $r^{\ell}$, respectively, where $k$ and $\ell$ are the
smallest non-negative integers such that $r^{k} \geq 2a$ and
$r^{\ell} \geq a$, respectively.

Observe (as in the previous two cases) that the Jordan canonical
form of $x$ viewed as a linear transformation on $V$ is a diagonal
matrix. Hence if $a_{1}, \ldots , a_{m}$ denotes the list of the
multiplicities of the distinct eigenvalues of $x$, then $r^{a} =
\sum_{i=1}^{m} a_{i}$. Again let $E$ be the enveloping algebra of
$R$. Since $x$ can be diagonalized on $V$, we certainly have $\dim
(C_{E}(x)) = \sum_{i=1}^{m} a^{2}_{i}$. However, $\dim (C_{E}(x))$
is {\it not} necessarily equal to the number $d$ of $\langle x
\rangle$-orbits of the set $R/Z(R)$.

Let us number the $\langle x \rangle$-orbits of the set of all
$r^{2a}$ vectors of $R/Z(R)$ from $1$ to $d$, and for all $1 \leq
i \leq d$ let $v_{i}$ be a representative of a coset in the $i$-th
orbit. If $\ell_{i}$ denotes the length of the $i$-th orbit, then
the elements $v^{x^j}_{i}$ form a set of coset representatives for
the cosets of $R/Z(R)$ where $i$ and $j$ run through the set of
numbers $1, \ldots , d$ and $1, \ldots , \ell_{i}$, respectively.

We claim that $\dim(C_{E}(x))$ is equal to the number of $i$'s for
which $v_{i} = {v_{i}}^{x^{\ell_{i}}}$.

For each $i$ let $E_{i}$ be the $\langle x \rangle$-invariant
subspace of $E$ generated by the vector $v_{i}$. It is easy to see
that $C_{E}(x) = \sum_{i=1}^{d}C_{E_i}(x)$. This implies that, in
order to prove the claim, it is sufficient to show that $\dim
(C_{E_i}(x)) = 1$ for all $1 \leq i \leq d$ for which $v_{i} =
{v_{i}}^{x^{\ell_{i}}}$ and that $\dim (C_{E_i}(x)) = 0$
otherwise. This is clear for those $i$'s for which $\ell_{i} = 1$.
It is also clear that if $i$ is so that $v_{i} =
{v_{i}}^{x^{\ell_{i}}}$, then $\dim(C_{E_{i}}(x)) \geq 1$. So let
$i$ be such that $\ell_{i} > 1$ and that $\dim(C_{E_{i}}(x)) > 0$.
Let $v \in C_{E_{i}}(x)$ be an arbitrary non-zero element. Write
$v$ in the form $\sum_{j=1}^{\ell_{i}} c_{j} v^{x^j}_{i}$ for some
field elements $c_{j}$ of $F$. Since $v$ is $\langle x
\rangle$-invariant and since the $v^{x^j}_{i}$'s are linearly
independent, it follows that all the $c_{j}$'s are equal. Hence
$C_{E_i}(x)$ is indeed $1$-dimensional. This proves our claim.

It remains to find an expression for $d$. Recall that there are
two cases we are interested in: if the Jordan canonical form of
$x$ considered as a linear transformation on the vector space
$R/Z(R)$ consists of a unique $2a$-by-$2a$ Jordan block or if it
consists of two $a$-by-$a$ Jordan blocks. In the first case let us
denote $d$ by $d_1$ while in the second case denote $d$ by $d_2$.

First suppose that the Jordan canonical form of $x$ consists of a
unique Jordan block. Then, as noted before, the order of $x$ is
$r^{k}$ where $k$ is the smallest positive integer such that
$r^{k} \geq 2a$. Every $\langle x \rangle$-orbit has prime power
length. It is easy to see that the number of orbits of length $1$
is $r = r^{\min \{ r^{0},2a \}}$, and for each $1 \leq i \leq k$
the number of orbits of length $r^i$ is $(1/r^{i}) \cdot ( r^{\min
\{ r^{i},2a \}}-r^{\min \{ r^{i-1},2a \} })$. This gives
\begin{equation}
\label{orbits1} d_{1} = r + \sum_{i=1}^{k} (1/r^{i}) \cdot
(r^{\min \{ r^{i},2a \}}-r^{\min \{ r^{i-1},2a \} }).
\end{equation}

By a similar argument, if the Jordan canonical form of $x$
consists of two $a$-by-$a$ Jordan blocks, then
\begin{equation}
\label{orbits2} d_{2} = r^{2} + \sum_{i=1}^{\ell} (1/r^{i}) \cdot
(r^{2 \min \{ r^{i},a \}}-r^{2 \min \{ r^{i-1},a \} })
\end{equation}
where $\ell$ is the smallest positive integer such that $r^{\ell}
\geq a$.

In the first case it is easy to see that $d_{1} \leq (1/4) \cdot
r^{2a}$ unless $a=1$ and $x$ has order $2$, $3$, $5$, or $7$. Let
$a=1$. If $x$ has order $2$, then $d(x) = \max_{i} \{ a_{i} \} =
1$. If $x$ has order $3$ or $5$, then $d(x) \leq 2$. If $x$ has
order $7$, then $d(x) \leq 3$. In all cases we have $d(x) \leq
((r+1)/2r) \cdot r^{a}$. In fact, we have $d(x) \leq (1/2) \cdot
r^{a}$ unless $a=1$, $r=3$ and $x$ has order $3$ in which case
$d(x) \leq (2/3) \cdot r^{a}$ holds. Let us summarize this result
in the following lemma.

\begin{lem}
\label{oneblock} Let us use the notations and assumptions of the
first paragraph of this section. Let $x$ be an element of $G
\setminus R$. Suppose that the Jordan canonical form of $x$ on
$R/Z(R)$ consists of a unique $2a$-by-$2a$ Jordan block, and that
$x$ has order a power of $r$. Then $d(x)/\dim_{F}(V) \leq
((r+1)/2r)$. Moreover we have $d(x) \leq (1/2) \cdot r^{a}$ unless
$a=1$, $r=3$ and $x$ has order $3$ in which case $d(x) \leq (2/3)
\cdot r^{a}$ holds.
\end{lem}

In the second case, $d_{2} \leq (1/4) \cdot r^{2a}$ unless $x$ has
order $4$ and $a=3$ or $a=4$, or $x$ has order $2$ and $a=2$, or
$x$ has order $3$ and $a=2$ or $a=3$. In all cases we will have
$d(x) \leq ((r+1)/2r) \cdot r^{a}$.

\begin{lem}
\label{twoblock} Let us use the notations and assumptions of the
first paragraph of this section. Let $x$ be an $r$-element in $G
\setminus R$, and let $R_1$, $R_2$ be two maximal abelian
subgroups of $R$ whose intersection is $Z(R)$. Suppose that the
Jordan canonical form of $x$ on $R/Z(R)$ consists of two
$a$-by-$a$ Jordan blocks that are the same where one leaves
$R_{1}/Z(R)$ invariant and the other leaves $R_{2}/Z(R)$
invariant. Then $d(x)/\dim_{F}(V) \leq ((r+1)/2r)$. Moreover we
have $d(x) \leq (1/2) \cdot r^{a}$ unless $a=2$, $r=2$ and $x$ has
order $2$ in which case $d(x) \leq (3/4) \cdot r^{a}$, or $a=2$,
$r=3$ and $x$ has order $3$ in which case $d(x) \leq (5/9) \cdot
r^{a}$.
\end{lem}

\begin{proof}
Let the order of the non-identity element $x$ be $q$. Since $q$ is
a power of $r$, the element $x$ can be diagonalized over $V$.
Suppose there are $m$ distinct eigenvalues. Let $a_i$ be the
multiplicity of the $i$-th eigenvalue. Then $d(x) = \max_{1 \leq i
\leq m} \{ a_{i} \}$. By the above, we have $r^{a} = \sum_{i} a_i$
and $d_{2} \geq \sum_{i} {a}^{2}_{i}$ where $d_2$ is as in
(\ref{orbits2}).

If $q=r$, then we can say even more. Indeed, in this case it is
easy to see that $r^{2a-1} + r^{2} - r = d_{2} =
\dim_{F}(C_{E}(x)) = \sum_{i} {a}^{2}_{i}$.

By the remark made just before the statement of the lemma, it is
sufficient to consider the following five cases.

Let $q=2$ and $a=2$. Then we have the equations $\sum_{i=1}^{m}
a_{i} = 4$ and $\sum_{i=1}^{m} {a}^{2}_{i} = 10$. Hence $d(x) =
3$.

Let $q=3$ and $a=2$. Then we have the equations $\sum_{i=1}^{m}
a_{i} = 9$ and $\sum_{i=1}^{m} {a}^{2}_{i} = 33$. We see that
$d(x) \leq 5$.

Let $q=3$ and $a=3$. Then $x$ acts on the set of distinct
eigenspaces of $R_1$ on $V$ having $8$ cycles of length $3$ and
$3$ fixed points. Hence $d(x) \leq 11$.

Let $q=4$ and $a=3$. Then $x$ acts on the set of distinct
eigenspaces of $R_1$ on $V$ having $1$ cycle of length $4$, $1$
cycle of length $2$, and $2$ fixed points. Hence $d(x) \leq 4$.

Let $q=4$ and $a=4$. Then $x$ acts on the set of distinct
eigenspaces of $R_1$ on $V$ having $3$ cycles of length $4$, $1$
cycle of length $2$, and $2$ fixed points. Hence $d(x) \leq 6$.
\end{proof}

Now let $x$ be {\it any} element of $G \setminus R$ of prime
order. We will use the previous lemmas of this section to show
that $d(x) \leq ((r+1)/2r) \cdot r^{a}$. For this purpose and for
the rest of this section we will use yet another lemma.

\begin{lem}[Guralnick, Malle, \cite{GM}]
\label{GM} Let $V_{1}$ and $V_{2}$ be an $F\langle x_{1} \rangle$-
and an $F\langle x_{2} \rangle$-module respectively for any field
$F$ and for group elements $x_{1}$ and $x_{2}$. Then $V_{1}
\otimes V_{2}$ can naturally be viewed as an $F\langle
(x_{1},x_{2}) \rangle$-module. Moreover, if $d(x_{1}) \leq c \cdot
\dim_{F}(V_{1})$ for some constant $c$, then $d((x_{1},x_{2}))
\leq c \cdot \dim_{F}(V_{1} \otimes V_{2})$.
\end{lem}

Put $c$ to be $(r+1)/2r$. By Lemmas \ref{oneblock} and 2.2, it is
easy to see that if $x$ acts indecomposably on $R/Z(R)$, then
$d(x) \leq c \cdot r^{a}$. So we may (and do) assume that $x$ does
not act indecomposably but decomposably on $R/Z(R)$. We claim that
$d(x) \leq c \cdot r^{a}$. We will argue by induction on the
number of indecomposable summands appearing in a direct sum
decomposition of the $\langle x \rangle$-module $R/Z(R)$.

First assume that the order of $x$ is coprime to $r$. Let $R_{1}$
be a minimal $\langle x \rangle$-invariant subspace in $R$. If
$R_{1}$ is non-degenerate, then, by 19.2 of \cite{A}, so is
${R_{1}}^{\bot}$ and $R = R_{1} \circ {R_{1}}^{\bot}$. By the
induction hypothesis, we conclude that $d(x) \leq c \cdot r^{a}$.
So we may (and do) suppose that $R_{1}$ is degenerate. By the
minimality of $R_{1}$, we have $R_{1} \subseteq R_{1}^{\bot}$
(since $R_{1} \cap R_{1}^{\bot}$ is a submodule of $R_{1}$). Since
$R$ is completely reducible, there exists an $\langle x
\rangle$-submodule $R_{2}$ of $R$ so that $R = R_{1}^{\bot}
R_{2}$. If $R_{2}$ is non-degenerate, then we can cook up the
decomposition $R = R_{2} \circ R_{2}^{\bot}$ and use the induction
hypothesis as before. So we may (and do) assume that both $R_{1}$
and $R_{2}$ are degenerate. Now $R_{2} \cap R_{2}^{\bot}$ is an
$\langle x \rangle$-submodule in the completely reducible module
$R_{2}$ so $R_{2} \subseteq R_{2}^{\bot}$ or there exists a
non-degenerate submodule $R_{3}$ such that $R_{2} = (R_{2} \cap
R_{2}^{\bot}) R_{3}$. In the latter case we may write $R = R_{3}
\circ R_{3}^{\bot}$ and apply the induction hypothesis to get the
desired conclusion. From now on we assume that both $R_{1}$ and
$R_{2}$ are degenerate and $R_{1} \subseteq R_{1}^{\bot}$, $R_{2}
\subseteq R_{2}^{\bot}$. Put $\widetilde{R} = R_{1} R_{2}$. We
claim that $\widetilde{R}$ is non-degenerate. We must show that
$\mathrm{Rad}(\widetilde{R}) \subseteq Z(R)$. Clearly,
$\mathrm{Rad}(\widetilde{R}) \subseteq R_{1}^{\bot}$. Since $R =
R_{1}^{\bot} R_{2}$, we have $R_{2} \cap R_{1}^{\bot} \subseteq
Z(R)$. The previous two statements imply $R_{2} \cap
\mathrm{Rad}(\widetilde{R}) \subseteq Z(R)$. From this it is not
difficult to conclude that $\mathrm{Rad}(\widetilde{R}) \subseteq
R_{1} \circ Z(R)$. This means that whenever $x \in
\mathrm{Rad}(\widetilde{R})$, then $[x,y] = 1$ for all $y \in
R_{1}^{\bot} \cup R_{2}$. Since $R = R_{1}^{\bot} R_{2}$, we
conclude that $Z(R) \subseteq \mathrm{Rad}(\widetilde{R})
\subseteq \mathrm{Rad}(R) \subseteq Z(R)$ which is exactly what we
wanted; $\widetilde{R}$ is indeed non-degenerate. If
$\widetilde{R} \not= R$, then $R = \widetilde{R} \circ
\widetilde{R}^{\bot}$ where both $\widetilde{R}$ and
$\widetilde{R}^{\bot}$ are non-degenerate and we may use the
induction hypothesis to get what we want. So we may assume that
$\widetilde{R} = R$.

Now assume that the order of $x$ is $r$.

For $r$ odd Hesselink \cite{Hesselink} showed that in the Jordan
normal form of $x$ on $R/Z(R)$ each `indecomposable part' consists
of a Jordan block of even size or of two Jordan blocks (of the
same) odd size. (Note that \cite[Remark, Page 172]{Hesselink}
points out that the field of order $r$ need not be quadratically
closed.) In the first case, the Jordan block of even size acts on
a non-degenerate space, while in the second case, the two Jordan
blocks act on totally singular subspaces. By our induction
hypothesis and \cite[19.2]{A}, we may assume that every
`indecomposable part' consists of two Jordan blocks of odd size.

If $r=2$ and $x$ is an involution then it is still true that in
the Jordan normal form of $x$ on $R/Z(R)$ each `indecomposable
part' consists of a single Jordan block (acting on a
non-degenerate space) or of two Jordan blocks (acting on totally
singular subspaces). This is because any $2$-dimensional subspace
of $R/Z(R)$ is either totally singular or non-degenerate with
respect to an alternating form. Again by our induction hypothesis
and \cite[19.2]{A}, we may assume that every `indecomposable part'
of $x$ consists of two Jordan blocks.


Write $R = (R_{1} R_{2}) \circ \ldots \circ (R_{\ell} R_{\ell+1})$
for some odd integer $\ell$, where, for all odd $1 \leq i \leq
\ell$, the cyclic group $\langle x \rangle$ acts indecomposably on
each of the two totally singular $\langle x \rangle$-modules
$R_{i}$ and $R_{i+1}$ with $R_{i} R_{i+1}$ acting absolutely
irreducibly on a vector space $V_{i}$ where $V = V_{1} \otimes
V_{3} \otimes \ldots \otimes V_{\ell}$. By Lemma \ref{GM}, it is
sufficient to show that for each odd $i$ the invariant $d(x)$ is
at most $c \cdot r^{a_{i}}$ where $r^{a_{i}} = \dim(V_{i})$. But
this follows from Lemmas \ref{HHtwoblocks} and \ref{twoblock}.

Let us summarize the results obtained so far in this section (with
bounds on $\dim_{F}(C_{V}(x))$ rather than on $d(x)$).

\begin{thm}
\label{dim} Let $V$ be a faithful irreducible $FG$-module where
$F$ is an algebraically closed field of characteristic $p > 0$ and
$G$ is a finite group. Suppose that $G$ has a normal subgroup $R$
of symplectic type with $|R/Z(R)| = r^{2a}$ for some prime $r$ and
that $R$ acts absolutely irreducibly on $V$, $\dim_{F}(V) = r^{a}$
and $O_{p}(G) = 1$. Suppose that $R$ is the unique normal subgroup
of $G$ that is minimal with respect to being non-central. Let $x$
be an arbitrary non-identity element in $G$. Then
$\dim_{F}(C_{V}(x)) \leq ((r+1)/2r) \cdot r^{a}$.
\end{thm}

\begin{proof}
If $1 \not= x \in R$, then $\dim_{F}(C_{V}(x)) \leq (1/2) \cdot
r^{a}$. If $x \in G \setminus R$ and $x$ has prime order, then
$\dim_{F}(C_{V}(x)) \leq ((r+1)/2r) \cdot r^{a}$. (These were
shown earlier.)

Finally, if $1 \not= x \in G$ is arbitrary and $q$ is a prime
proper divisor of the order $m$ of $x$, then $\dim_{F}(C_{V}(x))
\leq \dim_{F}(C_{V}({x}^{m/q})) \leq ((r+1)/2r) \cdot r^{a}$.
\end{proof}

However we will also need a more detailed result than Theorem
\ref{dim}. We start with a lemma.

\begin{lem}
\label{special} Let $E$ be a group of symplectic type with
$|E/Z(E)|=2^{2a}$. Let $V$ be an absolutely irreducible $E$-module
of dimension $2^{a}$. If $1 \ne x$ is a $2$-element in the
normalizer of $E$ in $GL(V)$ outside $E$ then one of the following
holds.
\begin{enumerate}
\item $\dim C_V(x) \le (1/2) \dim V$.

\item $x$ is an involution and in its action on $E/Z(E)$ each of
the $2m$ Jordan blocks of size $2$ act on totally singular
subspaces. (All other Jordan blocks have size $1$.) In this case
$\dim C_V(x) = (1/2)(1 + 2^{-m}) \dim V$.
\end{enumerate}
\end{lem}

\begin{proof}
Since any element of $E \setminus Z(E)$ has trace $0$ on $V$, the
fixed point space of any non-trivial element of $E$ has dimension
at most $(1/2) \dim V$. It suffices to assume that $x$ has order
$2$ or $4$. We may also assume that $\langle x \rangle \cap E=1$.

First suppose that $x$ is an involution. If $x$ leaves a
non-degenerate $2$-space invariant in its action on $E/Z(E)$, then
$x$ normalizes a non-abelian subgroup of order $8$, say $F$. Then
$V$ restricted to $J:=\langle F, x \rangle$ is a direct sum of
$2$-dimensional submodules (because $J/Z(J)$ is elementary abelian
of order $8$ and the derived subgroup of $J$ contains a
non-trivial central element of $E$). Since $x$ is an involution
and does not act trivially (since the normal closure of $x$ in $J$
contains the derived subgroup of $J$), $x$ has trace $0$, whence
the result.

Thus we may assume that all Jordan blocks of $x$ (in its action on
$E/Z(E)$) act on totally singular subspaces. Let the number of
Jordan blocks of size $2$ be $2m$. These form $m$ pairs acting on
the symplectic type $2$-groups $E_{1}, \ldots , E_{m}$ whose
central product with another symplectic type $2$-group $E_{0}$ is
$E$. (The element $x$ acts trivially on $E_{0}$.) Then the
$\langle x \rangle$-module $V$ has the form $V_{1} \otimes \cdots
\otimes V_{m} \otimes V_{0}$ where $V_{i}$ is an irreducible
$E_{i}$-module for every $i$ with $0 \leq i \leq m$. Now $x$ has
trace $2$ on each $V_{i}$ with $1 \leq i \leq m$ and trace $\dim
V_{0}$ on $V_{0}$. Hence the trace of $x$ on $V$ is $2^{m} \dim
V_{0}$ while $\mathrm{tr}(1) = \dim V = 4^{m} \dim V_{0}$. But
then
$$\dim C_{V}(x) = (1/2)(\mathrm{tr}(x) + \mathrm{tr}(1)) = (1/2) (1 + 2^{-m}) \dim V.$$


So now assume that $x$ has order $4$. Let $F/Z(E)$ be a
$3$-dimensional subspace of $E/Z(E)$ so that $x$ acts as a single
Jordan block on $F/Z(E)$. Suppose that $F/Z(E)$ is totally
singular (i.e. $F$ is abelian). Then $F$ has exactly $8$ distinct
eigenvalues on $V$ and $x$ permutes them in orbits of sizes $1$,
$1$, $2$, $4$, whence $\dim C_V(x) \le (1/2) \dim V$.

So we may assume that $F/Z(E)$ is not totally singular. Note that
since $[F,F] \cap Z(E) \ne 1$, it follows that every irreducible
$F$-submodule has dimension at least $2$. Since $F/Z(F)$ is
elementary abelian of order $8$, it follows that every absolutely
irreducible $F$-submodule (in any characteristic) has dimension at
most $2$ (it suffices to see this in characteristic $0$, but then
we know that every irreducible representation has dimension at
most $[F:Z(F)]^{1/2} < 3$). Thus, every irreducible $F$-submodule
(after extending scalars if necessary) of $V$ has dimension $2$.

Put $J:=\langle F, x \rangle$. Let $W$ be an irreducible
$J$-submodule of $V$. So $W$ is a direct sum of irreducible
$F$-submodules. If they are not isomorphic, then $x$ is permuting
the homogeneous components and so $\dim C_W(x) \le (1/2) \dim W$.
Suppose that $W$ is homogeneous as an $F$-submodule. Let $U$ be an
$F$-irreducible submodule of $W$. Thus, $W$ embeds in $X:=U_F^J$.
If $W$ is $2$-dimensional, then as the normal closure of $x$
contains $[F,F]$, $x$ is not trivial. Note that $\dim C_W(x) \le
\dim C_X(x) = 2$. Thus, the result holds for $\dim W > 2$. This
completes the proof.
\end{proof}

Let $x$ be the identity or an arbitrary element of $G \setminus
R$. View $x$ as a linear transformation on the vector space
$R/Z(R)$ and also as an element of $Sp(2a,r)$. We say that $x$ is
of type $B(2a,k)$ if the $GL(2a,r)$ Jordan canonical form of $x$
consists of two Jordan blocks each with minimal polynomial $f^{k}$
for some irreducible polynomial $f$ such that the $GL(2a,r)$
Jordan canonical form of $x^{q}$ (where $q$ is some power of $r$)
consists of $2k$ Jordan blocks that can be paired off in such a
way that each pair of blocks is of the kind treated in Lemma
\ref{HHtwoblocks}. Similarly, we say that $x$ is of type $C(2a)$
(or $D(2a)$) if the $GL(2a,r)$ Jordan canonical form of $x$ is the
kind treated in Lemma \ref{oneblock} (or Lemma \ref{twoblock}),
respectively. We say that $x$ has a part of type $B(2b,k)$,
$C(2b)$, or $D(2b)$ if there exists an $\langle x
\rangle$-invariant non-degenerate subspace $R_1$ of $R$ such that
the restriction of $x$ to $R_{1}/Z(R_{1})$ is of type $B(2b,k)$,
$C(2b)$, or $D(2b)$. Furthermore we say that $x$ has a part of
type $I_{2b}$ if there exists an $\langle x \rangle$-invariant
non-degenerate subspace $R_{1}$ of $R$ such that the restriction
of $x$ to the vector space $R_{1}/Z(R_{1})$ of order $r^{2b}$ is
the identity.

In the next two sections the following kinds of elements $x \in
Sp(2a,r)$ will play a fundamental role. In each case, using Lemma
\ref{special}, Lemmas \ref{HH}, \ref{HHtwoblocks}, \ref{oneblock},
\ref{twoblock}, Lemma \ref{GM}, and the argument above, we will
give an estimate for $\mathrm{rdim}(x) =
\dim_{F}(C_{V}(x))/\dim_{F}(V)$.

(i) Let $r=2$, $i$ be a positive integer at most $4$, and $2i \leq
a \leq 8$. The element $x$ has $i$ parts of type $D(4)$ and a part
of type $I_{2(a-2i)}$. $\mathrm{rdim}(x) = (1/2)(1+2^{-i})$.

If $r=2$, $a \leq 8$, and $x \not= 1$ is not of case (i) then
$\mathrm{rdim}(x) \leq 1/2$.

(ii) Let $r=3$ and $a \leq 4$. Let $i$ and $j$ be non-negative
integers with $1 \leq i+j \leq 4$ and $j \leq 1$. The element $x$
has $i$ parts of order $2$ of type $B(2,1)$, $j$ parts of type
$C(2)$, and a part of type $I_{2(a-i-j)}$. $\mathrm{rdim}(x) \leq
2/3$.

If $r=3$, $a \leq 4$, and $x \not= 1$ is not of case (ii) then
$\mathrm{rdim}(x) \leq 5/9$.

(iii) Let $r=5$ and $a \leq 2$. Let $i$ be $1$ or $2$ with $i \leq
a$. The element $x$ has $i$ parts of order $2$ of type $B(2,1)$
and a part of type $I_{2(a-i)}$. $\mathrm{rdim}(x) \leq 3/5$.

If $r=5$, $a \leq 2$, and $x \not= 1$ is not of case (iii) then
$\mathrm{rdim}(x) \leq 11/25$.

(iv) Let $r=7$ and $a = 1$. The element $x$ has a part of order
$2$ of type $B(2,1)$. $\mathrm{rdim}(x) \leq 4/7$.

If $r=7$, $a = 1$, and $x \not= 1$ is not of case (iv) then
$\mathrm{rdim}(x) \leq 3/7$.

\section{Counting certain elements}

In this section we are going to keep all the notations and
assumptions introduced in Section 2. In particular, let $R$, $F$,
$V$, and $G$ be as before.

For our future purposes we need to obtain an upper bound for the
number of non-identity elements $x \in G$ where
$\dim_{F}(C_{V}(x))$ is (relatively) large. We are only interested
in certain small cases, when $a \leq 8$ and $r=2$, $a \leq 4$ and
$r=3$, $a \leq 2$ and $r=5$, and when $a=1$ and $r=7$.

The factor group $G/R$ is isomorphic to a subgroup of the
symplectic group $Sp(2a,r)$ or to a subgroup of one of the
orthogonal groups $O^{\epsilon}(2a,2)$.

In this section let $L$ be one of the groups $Sp(2a,r)$ or
$O^{\epsilon}(2a,r)$ where $\epsilon$ is either $+$ or $-$. By
\cite{W}, \cite{FG} and \cite{ATLAS}, one can determine the number
of elements of $L$ with a given Jordan canonical form.

Much of the following is due to Wall \cite{W}, but we also follow
Fulman \cite{F}.

Let $K$ be the field with $r$ elements, and let $\phi(t) =
\alpha_{0} + \alpha_{1}t + \ldots + t^{deg(\phi)}$ be an
irreducible monic polynomial in $K[t]$ such that $\phi(0) =
\alpha_{0} \not= 0$. Define $\bar{\phi}(t)$ to be the monic
polynomial $({\alpha_{0}}^{-1})t^{deg(\phi)}\phi(t^{-1}) =
{\alpha_{0}}^{-1} + \ldots +
({\alpha_{0}}^{-1}\alpha_{1})t^{deg(\phi)-1} + t^{deg(\phi)}$.

Let $x$ be an element of $L$. Consider its Jordan canonical form.
To each power ${\phi}^{i}$ of an irreducible monic polynomial
$\phi$ one can associate a non-negative integer $m({\phi}^{i})$,
the multiplicity of a Jordan block with characteristic polynomial
${\phi}^{i}$. Similarly, to every irreducible monic polynomial
$\phi \not= t$ one can associate a partition $\lambda_{\phi}$ of
the non-negative integer $|\lambda_{\phi}|$ such that for each $i$
the number of parts equal to $i$ in $\lambda_{\phi}$ is
$m({\phi}^{i})$. For convenience, put $m_{i} = m({\phi}^{i})$. In
this way we can associate an (ordered) sequence $\Lambda_{x}$ of
$\ell(a,r)$ partitions to every element $x \in L$ where
$\ell(a,r)$ is the number of (monic) irreducible polynomials in
$K[t]$ different from $t$ whose degrees are no greater than $2a$.
(Note that $\emptyset$ is also considered to be a partition.) Fix
such a sequence of partitions $\Lambda$. In what follows we will
count $\mathcal{N}(\Lambda)$, the number of elements of $L$ whose
associated sequence of partitions is $\Lambda$.

Our first observation (probably due to Wall) is that
$\mathcal{N}(\Lambda) = 0$ unless $\Lambda$ is such that for all
irreducible polynomials $\phi$ different from $t$ we have
$\lambda_{\phi} = \lambda_{\bar{\phi}}$ and $\sum_{\phi \not= t}
|\lambda_{\phi}| deg(\phi) = 2a$.

The Jordan canonical form of $x \in L$ alone does not determine
the conjugacy class of $x$ in $L$.

In this paragraph let $r \not=2$. Wall \cite{W} showed that a
conjugacy class of $Sp(2a,r)$ corresponds to the following data.
To each monic, non-constant, irreducible polynomial $\phi \not= t
\pm 1$ associate a partition $\lambda_{\phi}$ (as before), and to
$\phi = t \pm 1$ associate a symplectic signed partition
$\lambda_{\phi}^{\pm}$, by which is meant a partition of some
natural number $|\lambda_{\phi}^{\pm}|$ such that the odd parts
have even multiplicity, together with a choice of sign for the set
of parts of size $i$ for each even $i
> 0$. These data represent a conjugacy class of $Sp(2a,r)$ if and
only if (1) $|\lambda_{t}| = 0$, (2) $\lambda_{\phi} =
\lambda_{\bar{\phi}}$ and (3) $\sum_{\phi \not= t}
|\lambda_{\phi}| deg(\phi) = 2a$.

Again, let $r \not= 2$. The orthogonal groups are the subgroups of
$GL(m,r)$ preserving a non-degenerate symmetric bilinear form. For
$m = 2l +1$ odd, there are two such forms up to isomorphism, with
inner product matrices $A$ and $\delta A$, where $\delta$ is a
non-square element of $K$ and $A$ is equal to
$$\left(
\begin{matrix}
1 &  0  & 0 \\
0 & 0_{l} & I_{l} \\
0 & I_{l} & 0_{l}
\end{matrix} \right).
$$
Denote the two corresponding orthogonal groups by $O^{+}(m,r)$ and
$O^{-}(m,r)$. This distinction will be useful, even though these
groups are isomorphic. For $m = 2l$ even, there are again two
non-degenerate symmetric bilinear forms up to isomorphism with
inner product matrices
$$\left(
\begin{matrix}
0_{l}  & I_{l} \\
I_{l} & 0_{l}
\end{matrix} \right)
$$
and
$$\left(
\begin{matrix}
0_{l-1}  & I_{l-1} & 0 & 0 \\
I_{l-1}  & 0_{l-1} & 0 & 0 \\
0 &  0 & 1 & 0 \\
0 & 0 & 0 & -\delta
\end{matrix} \right)
$$
where $\delta$ is a non-square element in $K$. Denote the
corresponding orthogonal groups by $O^{+}(2l,r)$ and
$O^{-}(2l,r)$. These groups are not isomorphic.

Consider the following combinatorial data. To each monic,
non-constant, irreducible polynomial $\phi \not= t \pm 1$
associate a partition $\lambda_{\phi}$ of some non-negative
integer $|\lambda_{\phi}|$ (as above), and to $\phi = t \pm 1$
associate an orthogonal signed partition $\lambda_{\phi}^{\pm}$,
by which is meant a partition of some natural number
$|\lambda_{\phi}^{\pm}|$ such that all even parts have even
multiplicity, and all odd $i > 0$ have a choice of sign. Wall
\cite{W} proved that these data represent a conjugacy class of
some orthogonal group if $r \not= 2$ or if $r=2$ and
$\lambda_{t+1} = \emptyset$.

From now on let $\Lambda$ be the associated sequence of partitions
of an element $x$ of $L$. For an irreducible polynomial $\phi$
such that $\lambda_{\phi} \not= \emptyset$, $\phi \not= t$ and
$\phi \not= t+1$ when $r=2$, define
$$B(\phi) = r^{deg(\phi) \big( \sum_{i < j}i m_{i} m_{j} + \frac{1}{2}
\sum_{i}(i -1)m^{2}_{i} \big)} \prod_{i} A({\phi}^{i}),$$ where

$$
A({\phi}^{i}) =
\begin{cases}
|Sp(m_{i},r)| & \text{if } i \equiv 1 (\bmod 2), r
\not=2, \phi = t \pm 1 \text{ and } L = Sp(2a,r); \\
r^{\frac{1}{2}m_{i}}|O(m_i,r)| & \text{if } i \equiv 0 (\bmod 2),
r \not=2, \phi = t \pm 1 \text{ and } L = Sp(2a,r); \\
|O(m_{i},r)| & \text{if } i \equiv 1 (\bmod 2), r \not=2,
\phi = t \pm 1 \text{ and } L = O^{\epsilon}(2a,r);\\
r^{-\frac{1}{2}m_{i}}|Sp(m_i,r)| & \text{if } i \equiv 0 (\bmod
2), r \not=2, \phi = t \pm 1 \text{ and } L = O^{\epsilon}(2a,r);\\
|U(m_{i},r^{deg(\phi)})| & \text{if } t \pm 1 \not=
\phi = \bar{\phi};\\
{|GL(m_{i},r^{deg(\phi)})|}^{1/2} & \text{if } t \pm 1 \not=
\phi \not= \bar{\phi}.\\
\end{cases}
$$

\noindent where, in the second and third cases above,
$|O(m_{i},r)|$ is $|O^{+}(m_{i},r)|$ if the sign chosen for the
parts equal to $i$ is $+$, and is $|O^{-}(m_{i},r)|$ if the sign
chosen for the parts equal to $i$ is $-$.

We are now in the position to state the first theorem of this
section.

\begin{thm}[Wall, \cite{W}]
\label{Wall} Use the notations of this section. If $\Lambda$ is
such that $\mathcal{N}(\Lambda) \not= 0$ and $\lambda_{t+1} =
\emptyset$ when $r=2$, then the number of elements in $L$ with
associated sequence of partitions $\Lambda$ is
$|L|/\prod_{\phi}B(\phi)$.
\end{thm}

We will demonstrate Theorem \ref{Wall} with two-three examples,
but only after the statement of Theorem \ref{FG}.

We also need some information on proportions of unipotent elements
in the group $L$ when $r=2$. The $GL(2a,r)$ Jordan canonical form
of a unipotent element in $L$ can be labelled by a partition $\mu$
of $2a$.

Let $\mu$ be a partition of a non-negative integer. Let $\mu_{i}$
be the $i$-th largest part of $\mu$, and let $o(\mu)$ be the
number of odd parts of $\mu$. The symbol $m_{i}$ will denote the
number of parts of $\mu$ of size $i$, and $\mu'$ is the partition
dual to $\mu$ in the sense that the $i$-th largest part $\mu_{i}'$
of $\mu'$ is $m_{i} + m_{i+1} + \ldots$. Let $n(\mu) = \sum_{i}
{\mu_{i}' \choose 2}$. Then we have

\begin{thm}[Fulman, Guralnick, \cite{FG}]
\label{FG} The number of elements of $Sp(2a,r)$ which are
unipotent and have $GL(2a,r)$ rational canonical form of type
$\mu$ is $0$ unless all odd parts of $\mu$ occur with even
multiplicity. If all odd parts of $\mu$ occur with even
multiplicity, it is
$$\frac{r^{a^{2}}\prod_{i=1}^{a}(r^{2i}-1)}
{ r^{n(\mu)+a + o(\mu)/2} \prod_{i}(1 - 1/r^{2}) \ldots (1 -
1/r^{2[m_{i}(\mu)/2]})}.$$
\end{thm}

Note that this theorem holds even when $r$ is a prime power,
however, we are only interested in the case when $r$ is a prime
and mostly when $r=2$.

Next we will give a few examples on how Theorems \ref{Wall} and
\ref{FG} can be used to estimate the numbers of certain kinds of
elements. For the reader's convenience we recall the formulas for
the orders of the classical groups we will be working with.

$$|Sp(2m,r)| = r^{m^2} \prod_{i=1}^{m}(r^{2i}-1);$$

$$|O^{\epsilon}(2m,r)| = 2 r^{m(m-1)} (r^{m}-\epsilon) \prod_{i=1}^{m-1}(r^{2i}-1);$$

$$|O^{\epsilon}(2m+1,r)| = 2r^{m} \prod_{i=0}^{m-1} (r^{2m} - r^{2i}).$$

\noindent {\it Example (i).} Using Theorem \ref{FG} we count
elements $x$ of $Sp(2a,2)$ of type (i) of Section 2. (The group
$Sp(2a,2)$ has two $2$-transitive permutation representations, one
with point-stabilizer $O^{+}(2a,2)$ and one with point-stabilizer
$O^{-}(2a,2)$. So the groups $O^{\epsilon}(2a,2)$ can be
considered as subgroups of $Sp(2a,2)$.) Here $\mu =
(2^{2i},1^{2a-4i})$, $o(\mu) = 2a - 4i$, $m_{2} = 2i$, $m_{1} = 2a
- 4i$, $\mu_{1}' = 2a - 2i$, $\mu_{2}' = 2i$, and $n(\mu) =
(a-i)(2a-2i-1) + i(2i-1)$. Thus the number of elements $x$ we are
looking for is (at most)
$$\frac{2^{i(i+1)}\prod_{j=1}^{a} (2^{2j}-1)}{\Big( \prod_{j=1}^{a-2i}(2^{2j}-1) \Big) \cdot \Big( \prod_{j=1}^{i}(2^{2j}-1) \Big)}.$$

\smallskip

\noindent {\it Example (ii).} Using Theorem \ref{Wall} we count
elements $x$ of $Sp(2a,3)$ of type (ii) of Section 2. If $\phi =
t+1$ then $m_{1} = 2i$ and $B(\phi) = |Sp(2i,3)|$ where $|Sp(0,3)|
= 1$. If $\phi = t-1$ then $m_{1} = 2(a-i-j)$, $m_{2} = j$ and
$$B(\phi) = 3^{2(a-i-j)j + (1/2)j(j+1)} |Sp(2(a-i-j),3)| \cdot |O^{\epsilon}(j,3)|$$
where $|O^{\epsilon}(0,3)| = 1$ and $\epsilon$ is the sign chosen
for the parts of size $2$. Hence the number of elements $x$ we are
looking for is (at most)
$$\Bigg( \frac{|Sp(2a,3)|}{3^{2(a-i-j)j + (1/2)j(j+1)} \cdot |Sp(2i,3)| \cdot |Sp(2(a-i-j),3)|} \Bigg) \Bigg(
\frac{1}{|O^{+}(j,3)|} + \frac{1}{|O^{-}(j,3)|} \Bigg).$$

\smallskip

\noindent {\it Example (iii).} Using Theorem \ref{Wall} we count
elements $x$ of $Sp(2a,r)$ of types (iii) and (iv) of Section 2.
(We have $r=5$ in the first case and $r=7$ in the second.) Suppose
first that $(r,a,i) \not= (5,2,1)$. Then $\phi = t+1$, $m_{1} =
2a$, and $B(\phi) = |Sp(2a,r)|$. Hence the number of such elements
$x$ is $1$. Now let $(r,a,i) = (5,2,1)$. Then the number of such
elements $x$ is $5^{2} (5^{2}+1)$.

\smallskip

Consider the table below. The star in a row corresponding to the
group $Sp(2a,r)$ stands for the positive integer $|Sp(2a,r)|$. Let
$A$ and $B$ be two consecutive entries in the row corresponding to
$Sp(2a,r)$. Suppose that $A$ (respectively $B$) lies in the column
corresponding to the fraction $c_{A}$ (respectively $c_{B}$).
(Clearly $c_{A} < c_{B}$.) Now $|R|B$ is an upper bound for the
number of elements $x$ in $G$ with $c_{A} < \mathrm{rdim}(x) \leq
c_{B}$.

\begin{center}
\begin{tabular}{|c|ccccccccccc|}

 \hline
  & $3/7$ & $11/25$ & $1/2$ & $17/32$ & $5/9$ & $9/16$ & $4/7$ & $3/5$ & $2/3$ &
     $5/8$ & $3/4$ \\
\hline
  $Sp(16,2)$ & & & * & $2^{72}$ & & $2^{67}$ & & & & $2^{53}$ & $2^{31}$ \\

  $Sp(14,2)$ & & & * & & & $2^{55}$ & &  & & $2^{45}$ & $2^{27}$ \\

  $Sp(8,3)$  & & & & & * & & & & $3^{29}$ & &  \\

  $Sp(12,2)$ & & & * & & & $2^{43}$ & & & & $2^{37}$ & $2^{23}$ \\

  $Sp(10,2)$ & & & * & & & & & & & $2^{29}$ & $2^{19}$ \\

  $Sp(6,3)$ & &  & & & * & & & & $3^{13}$ & & \\

  $Sp(4,5)$ & & * & & & & & & $651$ & & & \\

  $Sp(8,2)$ & & & * & & & & & & & $2^{21}$ & $2^{15}$  \\

  $Sp(4,3)$ & &  & & & * & & & & $982$ & &  \\

  $Sp(6,2)$ & & & * & & & & & & & & $2^{11}$ \\

  $Sp(2,7)$ & * & & & & &  & $1$ & & & & \\

  $Sp(2,5)$ & & * & & &  &  & & $1$ & & &  \\

  $Sp(4,2)$ & & & * & & & & & & & & $2^{7}$ \\

  $Sp(2,3)$ & & & & & * &  & &  & $10$ & &  \\

  $Sp(2,2)$ & & & * & & & & & & & & \\

\hline
\end{tabular}
\end{center}

\section{The proof of Theorem \ref{extraspecial}}

Let $k(X)$ denote the number of conjugacy classes of a finite
group $X$. This is also the number of complex irreducible
characters of $X$. We will use the following important result.

\begin{lem}
\label{lGT} Let $G$ be a group of linear transformations of the
finite vector space $V$, and let $GV$ be the semidirect product of
$V$ and $G$. Then $$k(GV) = \sum k(\mathrm{Stab}_{G}(\lambda))$$
where $\lambda$ is a complex irreducible character of $V$ and the
sum is over a set of representatives ($\lambda \in
\mathrm{Irr}(V)$) of the $G$-orbits of $\mathrm{Irr}(V)$.
\end{lem}

Notice that, in the previous two sections, our vector space was
finite dimensional over an algebraically closed field. If $F$ is a
finite subfield of an algebraically closed field $K$ and $V$ is an
$F\langle x \rangle$-module for some cyclic subgroup $\langle x
\rangle$, then, in a natural way, $V$ also has the structure of an
$K\langle x \rangle$-module where $\dim_{K}(V) = \dim_{F}(V)$.
Notice that we also have $\dim_{K}(C_{V}(x)) =
\dim_{F}(C_{V}(x))$.

In this section $G$ will denote a slightly different subgroup as
in the previous two sections. Let $r$ be a prime and let $R$ be an
$r$-group of symplectic type with $|R/Z(R)| = r^{2a}$ for some
positive integer $a$. Let $V$ be a faithful, absolutely
irreducible $KR$-module of dimension $r^{a}$ for some finite field
$K$. View $V$ as an $F$-vector space where $F$ is the prime field
of $K$. Let $G$ be a subgroup of $GL(V)$ which contains $R$ as a
normal subgroup. Let $A = C_{G}(K^{*})$.

The group $A$ is such that $A/R \leq Sp(2a,r)$ or $A/R \leq
O^{\epsilon}(2a,2)$ (latter only if $r=2$ and $|Z(R)|=2$) and
$|G/A| \leq k$ where $|K| = p^{k}$ and $p$ is prime. We have
$\dim(C_{V}(x))/\dim V \leq 1/2$ for every $x \in G \setminus A$.
(Indeed, there exists $z \in K^{*}$ with $1 \not= [x,z] \in K^{*}$
and thus $C_{V}([x,z]) = 1$. Hence $C_{V}(x^{-1}) \cap
C_{V}(x^{z}) = 1$. This implies $|V| \geq
|C_{V}(x^{-1})C_{V}(x^{z})| = |C_{V}(x^{-1})||C_{V}(x^{z})| =
{|C_{V}(x)|}^{2}$.) By this fact, by Theorem \ref{dim}, and by the
Orbit-Counting Lemma, the number of $G$-orbits on $V$ is at most
$(|V|/|G|) + |V|^{c}$ where $c = (r+1)/2r$. By Brauer's
Permutation Lemma, the number of $G$-orbits on $\mathrm{Irr}(V)$
is also at most $(|V|/|G|) + |V|^{c}$.

Let $m$ be the maximum of the $k(\mathrm{Stab}_{G}(\lambda))$'s as
$\lambda$ runs through the set of all non-trivial linear
characters of $V$. Then Lemma \ref{lGT} gives
\begin{equation}
\label{e1} k(GV) \leq k(G) + m \large( (|V|/|G|) + {|V|}^{c} -1
\large) .
\end{equation}
Let $\lambda$ be a non-trivial character in $\mathrm{Irr}(V)$.
Then $R \cap \mathrm{Stab}_{G}(\lambda)$ is an Abelian subgroup of
$R$ of order $r^{t}$ for some non-negative integer $t$ at most
$a$. Hence $|G:\mathrm{Stab}_{G}(\lambda)| \geq r^{2a+1-t}$. This
gives $m \leq |G|/r^{a+1}$. Applying these estimates to (\ref{e1})
we get
\begin{equation}
\label{e2} k(GV) \leq |G| + (|V|/r^{a+1}) + (|G|/r^{a+1}) \large(
{|V|}^{c} -1  \large) .
\end{equation}
It is possible to see that the right-hand-side of (\ref{e2}) is
less than $|V|$ unless $a \leq 8$ and $r=2$, $a \leq 4$ and $r=3$,
$a \leq 2$ and $r=5$, or $a=1$ and $r=7$. (Here we used the fact
that $r \mid (|K|-1)$ and that $|K| \geq 5$ in case $|Z(R)|=4$.)

Now let $(a,r)$ be such an exceptional pair.

For $r = 3$, $5$, $7$ let $c_{1} = 2/3$, $3/5$, $4/7$ and $c_{2} =
5/9$, $1/2$, $1/2$ respectively. By use of the table of the
previous section, we may give an upper bound $d_{1}$ for the
number of elements $x$ in $G$ with $c_{2} < \dim(C_{V}(x))/\dim V
\leq c_{1}$. Using this, Lemma \ref{lGT} and our bound for $m$, we
get
\begin{equation}
\label{e4}
\begin{split}
k(GV) & \leq k(G) + m \large( (|V|/|G|) + (d_{1}/|G|){|V|}^{c_{1}}
+ {|V|}^{c_{2}} \large) \\ & \leq |G| + (|V|/r^{a+1}) +
(d_{1}/r^{a+1}){|V|}^{c_{1}} + (|G|/r^{a+1}) {|V|}^{c_{2}}.
\end{split}
\end{equation}

For $r=2$ we use a slightly more detailed bound for $k(GV)$. For
each integer $i$ with $1 \leq i \leq 4$ let $d_{i}$ be the upper
bound (coming from the table of the previous section) for the
number of elements $x$ in $G$ with $$\dim(C_{V}(x))/\dim V =
(1/2)(1+2^{-i}).$$ Note that $d_{i} = 0$ whenever $a < 2i$. Then,
as before,
\begin{equation}
\label{f}
\begin{split}
k(GV) & \leq k(G) + m \large( (|V|/|G|) +
\sum_{i=1}^{4} \large( (d_{i}/|G|){|V|}^{(1/2)(1+2^{-i})} \large) + {|V|}^{1/2} \large) \\
& \leq |G| + (|V|/2^{a+1}) + \sum_{i=1}^{4} \large(
(d_{i}/2^{a+1}){|V|}^{(1/2)(1+2^{-i})} \large) + (|G|/2^{a+1})
{|V|}^{1/2}.
\end{split}
\end{equation}

Now let $a=8$ and $r=2$. By the last paragraph of Section 2 and
the table of Section 3, there are at most $2^{49}$ elements $x$ in
$G$ with $\mathrm{rdim}(x) = \dim (C_{V}(x))/\dim V = 3/4$, at
most $2^{71}$ elements $x$ with $\mathrm{rdim}(x) = 5/8$, at most
$2^{85}$ elements $x$ with $\mathrm{rdim}(x) = 9/16$, and at most
$2^{90}$ elements $x$ with $\mathrm{rdim}(x) = 17/32$. This
together with (\ref{f}) shows that $k(GV)$ is at most
$$|G| + (|V|/2^{9}) + 2^{40}{|V|}^{3/4} +
2^{62}{|V|}^{5/8} + 2^{76}{|V|}^{9/16} + 2^{81}{|V|}^{17/32} +
(|G|/2^{9}){|V|}^{1/2} \leq |V|.$$

Now let $a=7$ and $r=2$. By the last paragraph of Section 2, the
table of Section 3, and (\ref{f}) we have
$$k(GV) \leq |G| + (|V|/2^{8}) + 2^{35}{|V|}^{3/4} + 2^{53}{|V|}^{5/8} +
2^{63}{|V|}^{9/16} + (|G|/2^{8}){|V|}^{1/2} \leq |V|.$$

Now let $a=4$ and $r=3$. By the table of the previous section,
Theorem \ref{dim}, and (\ref{e4}) we have
$$k(GV) \leq |G| + (|V|/3^{5}) + (3^{33}){|V|}^{2/3}  + (|G|/3^{5}){|V|}^{5/9} \leq |V|.$$


Now let $a=5$ and $r=2$. By the last paragraph of Section 2, the
table of Section 3, and (\ref{f}) we have
$$k(GV) \leq |G| + (|V|/2^{6}) + 2^{25}{|V|}^{3/4} + 2^{35}{|V|}^{5/8}
+ (|G|/2^{6}){|V|}^{1/2}.$$ This is at most $|V|$ for $|V| \geq
17^{32}$. Notice that the possible prime divisors of $|G|$ are
$2$, $3$, $5$, $7$, $11$, $17$, $31$, and the prime divisors of
$k$. Thus by the (classical) $k(GV)$ theorem we have $k(GV) \leq
|V|$ unless $|K| = 3$, $5$, $7$, $9$, or $11$. In all these
exceptional cases we have $k(GV) \leq 2^{119}$.

Now let $a=3$ and $r=3$. By the table of the previous section,
Theorem \ref{dim}, and (\ref{e4}) we have
$$k(GV) \leq |G| + (|V|/3^{4}) + 3^{16}{|V|}^{2/3} + (|G|/3^{4}){|V|}^{5/9}.$$
This is at most $|V|$ for $|V| \geq 13^{27}$. Notice that the
possible prime divisors of $|G|$ are $2$, $3$, $5$, $7$, $13$, and
the prime divisors of $k$. Thus by the (classical) $k(GV)$ theorem
we have $k(GV) \leq |V|$ unless $|K| = 4$ or $7$. In all these
exceptional cases we have $k(GV) \leq 2^{82}$.

Now let $a=2$ and $r=5$. By the table of the previous section,
Theorem \ref{dim}, and (\ref{e4}) we have
$$k(GV) \leq |G| + (|V|/5^{3}) + 651 \cdot 5^{2}{|V|}^{3/5} + (|G|/5^{3}){|V|}^{1/2} \leq |V|.$$

Now let $a=4$ and $r=2$. By the last paragraph of Section 2, the
table of Section 3, and (\ref{f}) we have
$$k(GV) \leq |G| + (|V|/2^{5}) +  2^{20}{|V|}^{3/4} + 2^{26}{|V|}^{5/8}
+ (|G|/2^{5}){|V|}^{1/2}.$$ This is at most $|V|$ for $|V| \geq
41^{16}$. Notice that the possible prime divisors of $|G|$ are
$2$, $3$, $5$, $7$, $17$, and the prime divisors of $k$. Thus by
the (classical) $k(GV)$ theorem we have $k(GV) \leq |V|$ unless
$|K| = 3$, $5$, $7$, $9$, $17$, $25$, or $27$. In all these
exceptional cases we have $k(GV) \leq 2^{82}$.

Now let $a=2$ and $r=3$. By the table of the previous section,
Theorem \ref{dim}, and (\ref{e4}) we have
$$k(GV) \leq |G| + (|V|/3^{3}) + 8838 {|V|}^{2/3} + (|G|/3^{3}){|V|}^{5/9}.$$
This is at most $|V|$ for $|V| \geq 31^{9}$. Notice that the
possible prime divisors of $|G|$ are $2$, $3$, $5$, and the prime
divisors of $k$. Thus by the (classical) $k(GV)$ theorem we have
$k(GV) \leq |V|$ unless $|K| = 4$, $16$, or $25$. In all these
exceptional cases we have $k(GV) \leq 2^{44}$.

Now let $a=3$ and $r=2$. By the last paragraph of Section 2, the
table of Section 3, and (\ref{f}) we have
$$k(GV) \leq |G| + (|V|/2^{4}) + 2^{15}{|V|}^{3/4} +
(|G|/2^{4}){|V|}^{1/2}.$$ This is at most $|V|$ for $|V| \geq
191^{8}$. Notice that the possible prime divisors of $|G|$ are
$2$, $3$, $5$, $7$, and the prime divisors of $k$. Thus by the
(classical) $k(GV)$ theorem we have $k(GV) \leq |V|$ unless $|K| =
3$, $5$, $7$, $9$, $25$, $27$, $49$, $81$, or $125$. In all these
exceptional cases we have $k(GV) \leq 2^{58}$.


Now we turn to the treatment of the case $a=6$ and $r=2$. We need
a couple of lemmas.

\begin{lem}[Nagao, \cite{N}]
\label{lN} Let $N$ be a normal subgroup in a finite group $X$.
Then $k(X) \leq k(N) \cdot k(X/N)$.
\end{lem}

\begin{lem}
\label{root} Let $H$ be a subgroup of a finite group $X$. Then
$k(H) \leq \sqrt{|X| \cdot k(X)}$.
\end{lem}

Lemma \ref{root} is an easy consequence of a result of Gallagher
\cite{Ga} saying that $k(H) \leq (X:H) k(X)$.

Recall what $m$ and $k$ were above.

\begin{lem}
\label{6,2} Let $a=6$ and $r=2$. Then $m \leq 2^{50} k$.
\end{lem}

\begin{proof}
Let $\lambda$ be a non-trivial linear character of $V$ with
$k(\mathrm{Stab}_{G}(\lambda)) = m$. Put $T =
\mathrm{Stab}_{G}(\lambda)$. Since $R \cap T$ is normal in $T$, we
have $k(T) \leq k(R \cap T) \cdot k(T/(R \cap T))$ by Lemma
\ref{lN}, which is at most $r^{a} \cdot k(TR/R)$. By Lemma
\ref{lN} again, we see that $k(TR/R) \leq k \cdot k((TR \cap
A)/R)$. Now $H:= (TR \cap A)/R$ can be viewed as a subgroup of
$X:= Sp(12,2)$, and thus, by Lemma \ref{root}, we have $k(H) \leq
\sqrt{|X| \cdot k(X)} < 2^{44}$. (Here the estimate for $k(X)$
came from \cite[Theorem 3.13]{FulGur}.) Summing up, we have $k(T)
\leq 2^{50} k$.
\end{proof}

By the last paragraph of Section 2, the table of Section 3, and
(\ref{f}) we have that
$$k(GV) \leq |G| + (|V|/2^{7}) + 2^{34}{|V|}^{3/4} +
2^{52}{|V|}^{5/8} + 2^{62}{|V|}^{9/16} + m {|V|}^{1/2}.$$ By using
the bound of Lemma \ref{6,2} for $m$, we see that this is at most
$|V|$ provided that $|V| \geq 5^{64}$, and is less than $2^{120}$
if $|V| = 3^{64}$.

Finally we turn to the treatment of the cases $a=1$ and $(a,r) =
(2,2)$. The following lemma can be verified by GAP \cite{GAP}.

\begin{lem}
\label{a=1} Let $a=1$. If $r=7$, $5$, $3$, $2$, then $m$ is at
most $98 k$, $50 k$, $9 k$, $6k$ in the respective cases. If
$(a,r) = (2,2)$ then $m \leq 44k$.
\end{lem}

Now let $a=1$ and $r=7$. By the table of the previous section,
Theorem \ref{dim}, (\ref{e4}), and Lemma \ref{a=1} we have
$$k(GV) \leq |G| + (|V|/7^{2}) + 7{|V|}^{4/7} + m{|V|}^{1/2} \leq |V|.$$

Now let $a=1$ and $r=5$. By the table of the previous section,
Theorem \ref{dim}, (\ref{e4}), and Lemma \ref{a=1} we have
$$k(GV) \leq |G| + (|V|/5^{2}) + 5{|V|}^{3/5} + m{|V|}^{1/2} \leq |V|.$$

Now let $a=1$ and $r=3$. By the table of the previous section,
Theorem \ref{dim}, (\ref{e4}), and Lemma \ref{a=1} we have
$$k(GV) \leq k(G) + (|V|/3^{2}) + m {|V|}^{2/3},$$ which is at
most $|V|$ whenever $|V| \geq 13^3$. If $|V| = 7^3$ then $|G|$ is
coprime to $|V|$ and hence we have $k(GV) \leq |V|$ by the
(classical) $k(GV)$ theorem. If $|V| = 4^3$ then the inequality
$k(GV) \leq |V|$ can be checked by GAP \cite{GAP}.

Now let $a=1$ and $r=2$. By the last paragraph of Section 2, the
table of Section 3, (\ref{f}), and Lemma \ref{a=1} we have
$$k(GV) \leq k(G) + (|V|/2^{2}) + m{|V|}^{1/2},$$ which is at most
$|V|$ provided that $|V| \geq 13^{2}$. By the (classical) $k(GV)$
theorem, we may assume that the action of $G$ on $V$ is
non-coprime, that is, the cases remaining are $|V| = 3^2$ and $|V|
= 9^2$. In both these cases we have $k(GV) \leq \max \{ |V|, 11
\}$ by use of GAP \cite{GAP}.

Finally let $a=2$ and $r=2$. By the last paragraph of Section 2,
the table of Section 3, (\ref{f}), and Lemma \ref{a=1} we have
$$k(GV) \leq |G| + (|V|/2^{3}) +  44 k {|V|}^{3/4}.$$ This is at most $|V|$ for $|V|
> 243^{4}$. Notice that the possible prime divisors of $|G|$ are
$2$, $3$, $5$, and the prime divisors of $k$. Thus by the
(classical) $k(GV)$ theorem we have $k(GV) \leq |V|$ unless $|K| =
3$, $5$, $9$, $25$, $27$, $81$, $125$, or $243$. In all these
exceptional cases we have $k(GV) \leq 2^{32}$.

This completes the proof of Theorem \ref{extraspecial}.

\section{Another estimate for $k(GV)$}

Let $t \geq 2$ be a positive integer. For all integers $i$ such
that $1 \leq i \leq t$ define $r_{i}$, $R_{i}$, $a_{i}$, $V_{i}$,
$G_{i}$ in the following way. Let $r_{i}$ be a prime and let
$R_{i}$ be an $r_{i}$-group of symplectic type with
$|R_{i}/Z(R_{i})| = r_{i}^{2a_{i}}$ for some positive integer
$a_{i}$. Suppose also that $r_{i} \not= r_{j}$ whenever $i \not=
j$. Let $V_{i}$ be a faithful, absolutely irreducible
$KR_{i}$-module of dimension $r_{i}^{a_{i}}$ for some finite field
$K$. View $V_{i}$ as an $F$-vector space where $F$ is the prime
field of $K$. Let $G_{i}$ be a subgroup of $GL(V_{i})$ which
contains $R_{i}$ as a normal subgroup. Let $G$ be the central
product of the $G_{i}$'s with $Z$ for some group of scalars $Z$.
(Put $A = C_{G}(K^{*})$.) Then the vector space $V = V_{1}
\otimes_{F} \cdots \otimes_{F} V_{t}$ can be considered as an
$FG$-module.

In this section we will bound $k(GV)$ using Lemma \ref{lGT}.

Let $R$ be the central product of all the $R_{i}$'s and $Z$.
Moreover put $n = \prod_{i=1}^{t} r_{i}^{a_{i}}$ which is the
$K$-dimension of the vector space $V$.

\begin{lem}
\label{segedlemma} Use the notations of this section. If $\lambda$
is a non-trivial linear character of $V$ then
$|\mathrm{Stab}_{G}(\lambda)| \leq |G/R| \cdot n$.
\end{lem}

\begin{proof}
The inertia group $\mathrm{Stab}_{R}(\lambda)$ of a non-trivial
linear character $\lambda$ of $V$ in $R$ is abelian since it
embeds injectively into the abelian group $R/Z(R)$. For all $i$
the subgroup $\mathrm{Stab}_{R}(\lambda) \cap R_{i}$ is abelian of
order at most $r_{i}^{a_{i}}$. Hence $|\mathrm{Stab}_{G}(\lambda)|
\leq |G/R||\mathrm{Stab}_{R}(\lambda)| \leq |G/R| \cdot n$.
\end{proof}

By \cite[Pages 82-83]{MW} we have $\dim(C_{V}(x))/\dim V \leq 3/4$
for all non-identity elements $x$ in $G$. This and Lemma \ref{lGT}
imply (as in the previous section) that
$$k(GV) \leq |G| + m ((|V|/|G|) + {|V|}^{3/4})$$
where $m$ is the maximum of the $k(\mathrm{Stab}_{G}(\lambda))$'s
as $\lambda$ runs through the set of non-trivial linear characters
of $V$. By Lemma \ref{segedlemma} this number is at most
$|G/R|\cdot n$, hence we get
\begin{equation}
\label{e3} k(GV) \leq |G| + ((|V| \cdot n) / |R|) + ((|G| \cdot n)
/|R|) {|V|}^{3/4}.
\end{equation}
Now $n^{2} \leq |R| \leq n^{3} |K|$ and
$$|G/R| \leq k \cdot \prod_{i=1}^{t} |Sp(2a_{i},r_{i})| \leq k \cdot n^{3 \log_{2} n}$$
where $|K| = p^{k}$ and $|F| = p$. Hence (\ref{e3}) gives
$$k(GV) \leq p^{k} \cdot k \cdot n^{3 + 3 \log_{2} n} + (|V|/n) + p^{k} \cdot k \cdot n^{1 + 3 \log_{2}n} {|V|}^{3/4}$$
which is, by inspection, at most $\max \{ |V|, 2^{1344} \}$.

Summarizing the content of this section with Theorem
\ref{extraspecial} we get the following.

\begin{thm}
\label{section5} Use the notations of this section with allowing
$t=1$. Then $k(GV) \leq \max \{ |V|, 2^{1344} \}$.
\end{thm}

\section{The meta-cyclic case}

In this section we prove Theorem \ref{main}.

Let $p$ be a prime. Let $X$ be $GL(1,p^{n}).n$ for some positive
integer $n$. Then $X$ has a maximal abelian normal subgroup $S$
which is cyclic of order $p^{n}-1$ and $|X| = n (p^{n}-1)$.
Furthermore, $X$ is meta-cyclic and any element $x$ of $X$ can be
written in the form $x = a^{k}b^{l}$ for some integers $k$ and $l$
with $0 \leq k < p^{n}-1$ and $0 \leq l < n$ where $\langle a
\rangle = S$, $\langle bS \rangle = X/S$, and $a^{p}b = ba$.

Let $G$ be a subgroup of $X$. Then $G/(S \cap G)$ is cyclic of
order $d$ for some $d$ dividing $n$ and $1 \leq d \leq n$. Suppose
that $S \cap G = \langle a^{m} \rangle$ where $m$ is an integer
with $0 < m \leq p^{n}-1$ and is as small as possible. By our
choice of $m$ the integer $p^{n}-1$ is divisible by $m$. Let $c
\in G$ so that $c(S \cap G)$ generates $G / (S \cap G)$. Then
there exists an integer $k$ with $0 \leq k < p^{n}-1$ so that $c$
is of the form $a^{k}b^{n/d}$.

The group $G$ acts in a natural way on the $n$-dimensional vector
space $V$ over $GF(p)$. Let the semidirect product of $G$ with the
abelian (additive) group $V$ be $GV$. Let us view $V$ as a field
of order $p^n$ and let $a_{0}$ be a generator of the
multiplicative group of $V$ so that the equations
$b^{-1}a_{0}^{t}b = a_{0}^{tp}$ and $a^{-1}a_{0}^{t}a =
a_{0}^{t+1}$ hold for every integer $t$ with $0 \leq t < p^{n}-1$.

In order to prove Theorem \ref{main} we wish to bound the number
$k(GV)$ of complex irreducible characters of the group $GV$. By
\cite{GAP}, Theorem \ref{main} can be verified for all prime
powers $p^{n}$ at most $1024$. Hence from now on in our
considerations we will assume that $p^{n} > 1024$.

We will use several lemmas to show Theorem \ref{main}.




\begin{lem}[Gallagher, \cite{Ga}]
\label{l2} Let $H$ be a finite group, $N$ be a normal subgroup in
$H$, $\chi$ be an irreducible character of $N$, and $I(\chi)$ be
its inertia subgroup. Then the number of irreducible characters of
$H$ which lie over ($H$-conjugates of) $\chi$ is at most
$k(I(\chi)/N)$.
\end{lem}

Clearly, to apply the above lemmas, we are interested in the
$G$-orbits of the set $\mathrm{Irr}(V)$. In case $|G|$ is not
divisible by $p$, then $\mathrm{Irr}(V)$ and $V$ are permutation
isomorphic $G$-sets, however this is not always so in case $p$
divides $|G|$. In any case, the following two consequences of
Brauer's Permutation Lemma (\cite[Theorem 6.32]{Isaacs}) will be
used.

\begin{lem}
\label{l3} The number of $G$-orbits on $\mathrm{Irr}(V)$ is equal
to the number of $G$-orbits on $V$.
\end{lem}

\begin{lem}
\label{l4} Let $H$ be a finite group, $N$ be an abelian normal
subgroup in $H$, and suppose that $H/N$ is cyclic. Then there is a
size preserving bijection between the set of $H/N$-orbits of
$\mathrm{Irr}(N)$ and the set of $H/N$-orbits of $N$.
\end{lem}

We may assume that $d > 1$. Indeed, if $d=1$, then $G = S \cap G$,
and so $G$ acts semiregularly on the set of non-zero vectors of
$V$. Since $p$ does not divide $|G|$, the $G$-sets $V$ and
$\mathrm{Irr}(V)$ are permutation isomorphic. Hence, by Lemma
\ref{lGT}, $k(GV) = k(G) + m = ((p^{n}-1)/m) + m \leq p^{n}$ which
is exactly what we wanted.

From now on assume that $d > 1$ and let $q$ be the smallest prime
divisor of $d$.

\begin{lem}
\label{l5} With the above notations and assumptions we have
$$k(G) \leq \frac{q^{2}-1}{q^{2}}d(p^{n/d}-1) +
\frac{d(p^{n}-1)}{q^{2}m}.$$
\end{lem}

\begin{proof}
We see that
$$[a^{m},c^{-1}] = [a^{m},b^{-n/d}] = a^{-m}b^{n/d}a^{m}b^{-n/d} = a^{m(p^{n/d}-1)} \in G'.$$
Hence $|G'| \geq \frac{p^{n}-1}{m(p^{n/d}-1)}$ and so $|G/G'| \leq
d(p^{n/d}-1)$. This means that the number of linear complex
irreducible characters of $G$ is at most $d(p^{n/d}-1)$. By the
fact that $\sum_{\chi \in \mathrm{Irr}(G)} {\chi(1)}^{2} = |G|$
and by Ito's Theorem (\cite[Corollary 6.15]{Isaacs}) we have
$$|\mathrm{Irr}(G)| \leq d(p^{n/d}-1) + \frac{|G|-d(p^{n/d}-1)}{q^{2}} = \frac{q^{2}-1}{q^{2}}d(p^{n/d}-1) +
\frac{d(p^{n}-1)}{q^{2}m}.$$
\end{proof}

\begin{lem}
\label{lA} Use the above notations and assumptions. Let $g$ be a
non-identity element of $G$. Then $|C_{V}(g)| \leq |V|^{1/q}$.
\end{lem}

\begin{proof}
Let $g = a^{l}b^{r(n/d)}$ be a non-identity element of $G$ for
some integers $l$ and $r$ with $0 \leq l < p^{n}-1$ and $0 \leq r
< d$. We may assume that $0 < r$.

Let $t$ be an integer with $0 \leq t < p^{n}-1$ so that
${a_{0}}^{t}$ is centralized by $g$. Then
$$a_{0}^{t} = g^{-1}a_{0}^{t}g = b^{-r(n/d)}a^{-l}{a_{0}}^{t}a^{l}b^{r(n/d)}
= b^{-r(n/d)}{a_{0}}^{t+l}b^{r(n/d)} =
{a_{0}}^{(t+l)p^{r(n/d)}}.$$ This implies that $t \equiv
(t+l)p^{r(n/d)} \pmod {p^{n}-1}$, that is, $$(p^{r(n/d)}-1)t
\equiv -l p^{r(n/d)} \pmod {p^{n}-1}.$$

Let $s_{0}$ be the smallest positive integer $s$ so that $p^{n}-1
\mid s (p^{r(n/d)}-1)$. Then $s_{0} \mid p^{n}-1$.

It is easy to see that for any integer $x$ there is either no
solution $t$ to the congruence
$$(p^{r(n/d)}-1)t \equiv x \pmod {p^{n}-1}$$ or there are exactly
$(p^{n}-1)/s_{0}$ solutions $t$ in the range $0 \leq t < p^{n}-1$.
For $x=0$ there is a solution hence there must be
$(p^{n}-1)/s_{0}$. So in order to maximize $|C_{V}(g)|$ we may
assume that $s_{0} = 1$ and thus $-l p^{r(n/d)} \equiv 0 \pmod
{p^{n}-1}$, that is, $l = 0$ and thus $g = b^{r(n/d)}$. But in
this case $C_{V}(g)$ can be considered as a subfield of $V$ of
order at most $p^{n/q}$ (where $q$ is the smallest prime divisor
of $d$).
\end{proof}

\begin{lem}
\label{lB} The number of $G$-orbits on $\mathrm{Irr}(V)$ is at
most $(|V|/|G|) + |V|^{1/q}$.
\end{lem}

\begin{proof}
This follows from Lemmas \ref{l3} and \ref{lA}.
\end{proof}

We may assume that $m < p^{n}-1$. For otherwise $|G|=d$ and $G =
\langle c \rangle$. By Lemmas \ref{l4}, \ref{lA}, and \ref{lB},
there are at most $|V|^{1/q}$ $G$-invariant irreducible characters
of $V$ and all other complex irreducible characters of $V$ lie in
$G$-orbits of lengths at least $q$. By Lemmas \ref{l2} and
\ref{lB} we find that
$$k(GV) \leq \frac{d}{q} \frac{|V|}{d} + d |V|^{1/q} = \frac{|V|}{q} + d|V|^{1/q} \leq
|V|$$ unless $p^{n} \leq 64$ (but we assumed that $p^{n} > 1024$).

\begin{lem}
\label{l6} Let $\chi$ be an irreducible character of $V$. Then
$I(\chi)/V$ is a cyclic group of order at most $d$.
\end{lem}

\begin{proof}
This follows from the fact that $(I(\chi)/V) \cap (S \cap G)$ is
trivial since $S \cap G$ acts semiregularly on $V \setminus \{ 1
\}$ and hence acts semiregularly on $\mathrm{Irr}(V) \setminus \{
1 \}$ (since $V$ and $\mathrm{Irr}(V)$ are permutation isomorphic
$S \cap G$ sets).
\end{proof}

\begin{lem}
\label{l7} Use the above notations and assumptions. We have
$$k(GV) \leq \frac{q^{2}-1}{q^{2}}d(p^{n/d}-1) + \frac{(p^{n}-1)d}{q^{2}m} + \frac{p^{n}}{p^{n}-1}m + dp^{n/q}.$$
\end{lem}

\begin{proof}
By Lemmas \ref{l2}, \ref{lB}, and \ref{l6} we have $k(GV) \leq
k(G) + d ((|V|/|G|) + |V|^{1/q})$. Finally, the claim follows from
Lemma \ref{l5}.
\end{proof}

Using Lemma \ref{l7} and our assumptions including the hypothesis
that $p^{n} > 1024$, we may assume that $m < d$. Indeed, if $m =
(p^{n}-1)/2$, then $k(GV) \leq |V|$. (Cases $d=2$, $d=3$, and $d
\geq 4$ should be treated separately.) Also, if $m = (p^{n}-1)/3$,
then $k(GV) \leq |V|$. Finally, if $d \leq m \leq (p^{n}-1)/4$,
then $k(GV) \leq |V|$. (Cases $d=2$ and $d \geq 3$ should be
treated separately.)

Let $n=2$. Then $d = 2 = n$. Since $m < d$, we have $m=1$, and so
$G = X$. There are $p-1$ fixed points of $\langle b \rangle$ on
$S$ and $(p^{2}-p)/2$ orbits of length $2$. Hence, by Lemmas
\ref{l4} and \ref{l2}, we have $k(G) \leq 2(p-1)+ (p^{2}-p)/2$.
This and the proof of Lemma \ref{l7} gives
$$k(GV) \leq 2(p-1) + \frac{p^{2}-p}{2} + \frac{p^2}{p^{2}-1} + 2p$$
which is at most $p^{2}$ (provided that $p^{2} > 1024$).

By the previous paragraph we may assume that $n > 2$.

By Zsigmondy's Theorem, there is a primitive prime divisor $p_{n}$
of $p^{n}-1$. By a primitive prime divisor we mean a prime which
divides $p^{n}-1$ but divides none of the integers $p^{r}-1$ where
$1 \leq r < n$. Such a primitive prime divisor is congruent to $1$
modulo $n$, in particular, $p_{n} \geq n+1$. Since $m < d \leq n <
p_{n}$, the prime $p_{n}$ does not divide $m$.

For every integer $t$ with $1 \leq t \leq (p^{n}-1)/m$ and for
every integer $r$ with $1 \leq r \leq d$ we have
$$c^{r}a^{mt}c^{-r} = a^{p^{nr/d}mt}.$$
If $p_{n} \nmid t$, then $p_{n} \nmid mt(p^{nr/d}-1)$ for $r< d$,
and so $p^{n}-1 \nmid mt(p^{nr/d}-1)$ for $r < d$. This means that
$a^{mt}$ is not fixed by any element $c^{r}$ for $r < d$. We
conclude that if $p_{n} \nmid t$, then $a^{mt}$ lies in a $\langle
c \rangle$-orbit of $S \cap G$ of length $d$. There are at most
$(p^{n}-1)/mp_{n}$ elements of $S \cap G$ which lie in $\langle c
\rangle$-orbits of lengths less than $d$. Hence there are at most
$$\frac{p^{n}-1}{m p_{n}} + \frac{p^{n}-1}{m d} - \frac{p^{n}-1}{m d p_{n}}$$
$\langle c \rangle$-orbits of $S \cap G$.

\begin{lem}
\label{l8} By the above notations and assumptions including $m <
d$, we have
$$k(GV) \leq \frac{d(p^{n}-1)}{m p_{n}} + \frac{p^{n}-1}{m d}
- \frac{p^{n}-1}{m d p_{n}} + \frac{p^{n}}{p^{n}-1}m + d
p^{n/q}.$$
\end{lem}

\begin{proof}
This follows from Lemmas \ref{l4}, \ref{l2}, and \ref{lB}.
\end{proof}

By Lemma \ref{l8}, we may assume that $m=1$. This follows by
treating the three cases $m \geq 3$, $m = 2$ and $d = 2$, and $m =
2$ and $d \geq 3$ separately. Since $m=1$, we may assume that $c =
b^{n/d}$.

\begin{lem}
\label{l9} If $m=1$, then
$$k(GV) < \sum_{r \mid d} \Big( \frac{d}{r^{2}} p^{nr/d}\Big) + 2 + d p^{n/q}.$$
\end{lem}

\begin{proof}
By Lemma \ref{l8} and its proof it is sufficient to show that
$$k(G) < \sum_{r \mid d} \Big( \frac{d}{r^{2}} p^{nr/d}\Big).$$
For any positive integer $r$ dividing $d$ and any integer $t$ the
element $a^{t}$ of $S$ is fixed by $c^{-r}$ if and only if
$p^{n}-1 \mid t(p^{nr/d}-1)$. Hence there are less than
$p^{nr/d}/r$ orbits of length $r$ of $\langle c \rangle$ on $S$.
Now apply Lemma \ref{l4} and Lemma \ref{l2} to obtain the desired
conclusion.
\end{proof}

Using Lemma \ref{l9} and the assumption that $p^{n} > 1024$, we
find that $k(GV) \leq |V|$ for $d$ a prime and for $d = 4$, $6$,
$8$, and $9$. Hence we may assume that $d \geq 10$.

Finally, again by Lemma \ref{l9}, if $d \geq 10$ and $p^{n} >
1024$, we have
$$k(GV) < \frac{p^{n}}{d} + \frac{4p^{n/2}}{d} +
\frac{9 \cdot p^{n/3}}{d} + n^{2}p^{n/4} + 2 + n \cdot p^{n/2}
\leq p^{n}.$$

This finishes the proof of Theorem \ref{main}.

\section{Primitive linear groups}

Let $V$ be a finite faithful irreducible primitive $FG$-module.
Suppose that the generalized Fitting subgroup of $G$ is nilpotent.
In this section we will prove the estimates $k(GV) \leq \max \{
|V|, 2^{1344} \}$. This will verify Theorem \ref{primitive}.

Put $q=|F|$ and let $n$ be the $F$-dimension of $V$. We may
consider $G$ as a primitive (irreducible) subgroup of $GL(n,q)$.
We may also assume that $G$ is absolutely irreducible. (For
suppose that $F \not= K:= \mathrm{End}_{FG}(V)$. Then $V$ can be
viewed as a $K$-vector space of dimension $n/|K:F|$. In fact $V$
is an absolutely irreducible, primitive and faithful $KG$-module.
So if we prove the bound for $k(GV)$ where $V$ is a $KG$-module,
then the same bound will hold when $V$ is an $FG$-module.)



Let us recall a consequence of Clifford's theorem. A normal
subgroup of a primitive (irreducible) linear group acts
homogeneously on the underlying vector space. This means that any
two simple submodules of the normal subgroup are isomorphic. One
importance of this observation is that if $N \leq GL(n,q)$ is a
normal subgroup of a primitive group, then $N$ is irreducible or
$N$ can be considered to be an irreducible subgroup of $GL(d,q)$
where $d< n$ and $d \mid n$.



First suppose that whenever $N$ is a normal subgroup of $G$, then
every irreducible $FN$-submodule of $V$ is absolutely irreducible.


As we have noted before, every normal subgroup of $G$ acts
homogeneously on $V$. In particular, any abelian normal subgroup
acts homogeneously, and so is cyclic by Schur's lemma.
Furthermore, by our hypothesis on the normal subgroups of $G$, an
abelian normal subgroup of $G$ must be central (in $G$).

If all normal subgroups of $G$ are central, then $G$ is abelian,
$n=1$ and $|G| \leq q^{n}-1$. In this case $k(GV) \leq |V|$.

From now on suppose that $G$ has at least one non-central normal
subgroup.

Let $R$ be a normal subgroup of $G$ that is minimal with respect
to being non-central. We see that $R$ is non-abelian. Since $Z(R)$
is abelian and normal in $G$, we have $Z(R) \leq Z(G)$. By the
minimality of $R$, the factor group $R/Z(R)$ is characteristically
simple. So either $R$ is a central product of say $\ell$
quasi-simple groups $Q_i$ (with $Q_{i}/ Z(Q_{i})$ all isomorphic),
or $R/Z(R)$ is an elementary abelian $r$-group for some prime $r$.
The previous case does not occur since we assumed that $G$ has no
non-abelian component. In the latter case it follows that $R$ is
of symplectic type with $|R/Z(R)| = r^{2a}$ for some prime $r$ and
some positive integer $a$.

Let $J_{1}, \ldots , J_{t}$ denote the distinct normal subgroups
of $G$ that are minimal with respect to being non-central in $G$.
Put $J = J_{1} \cdots J_{t}$. Then $C_{G}(J) = Z(G)$. (The
containment $C_{G}(J) \supseteq Z(G)$ is clear. Suppose that
$C_{G}(J)$ properly contains $Z(G)$. Since $C_{G}(J)$ is a normal
subgroup of $G$ which is not central, $C_{G}(J)$ contains a normal
subgroup of $G$ which is minimal with respect to being
non-central. Without loss of generality, let such a subgroup be
$J_{1}$. Then $J_{1}$ must be abelian and so central (by the fifth
paragraph of this section).) Thus, $G/Z(G)J$ embeds into the
direct product of the outer automorphism groups of the minimal
normal non-central subgroups. If $R$ is of symplectic type with
$|R/Z(R)| = r^{2a}$, then this outer automorphism group is
isomorphic to $Sp(2a,r)$ or to $O^{\epsilon}(2a,2)$ in case $r=2$
and $|Z(R)|=2$.


Since $G$ is primitive on $V$, the normal subgroup $J$ acts
homogeneously on $V$. Let $W$ be an irreducible constituent for
$J$. It follows that $W \cong U_{1} \otimes \cdots \otimes U_{t}$
where $U_{i}$ is an irreducible $FJ_{i}$-module. Furthermore, $W$
is an irreducible faithful $FG$-module as the one treated in
Section 5; hence Theorem \ref{section5} applies.

\begin{lem}
\label{l33} Use the notations and assumptions of this section.
Suppose that $G$ is an absolutely irreducible subgroup of
$GL(n,q)$ and that whenever $N$ is a normal subgroup of $G$, then
every irreducible $FN$-submodule of $V$ is absolutely irreducible.
Then $k(GV) \leq \max \{ |V|, 2^{1344} \}$.
\end{lem}

There are two cases remaining: $G$ either preserves a field
extension structure on $V$ or it does not. (We say that $G$
preserves a field extension structure on $V$ if there is an
$F$-subalgebra $K \subseteq \mathrm{End}_{F}(V)$ (properly
containing $F$) so that $G$ preserves $K$ (and there is a
homomorphism from $G$ into $\mathrm{Gal}(K | F)$).)

Suppose that $G$ preserves no field extension structure. Let $N$
be a normal subgroup of $G$. Then $N$ must act homogeneously on
$V$ (since $G$ acts primitively on $V$) and moreover, the
irreducible constituents for $N$ must be absolutely irreducible
(otherwise the center of $\mathrm{End}_{N}(V)$ is $K$ for some
field extension of $F$, and would be normalized by $G$, whence $G$
preserves a field extension structure on $V$). Hence, in this
case, Lemma \ref{l33} gives the desired conclusion.

Now suppose that $G$ preserves a field extension structure on $V$
over a field $K$ with $K$ as large as possible. Let $|K|=q^{e}$
with $e>1$. Let $A = G \cap GL(n/e,q^{e})$. Let $U$ denote $V$
considered as a vector space over $K$ (and as a $KA$-module). Then
$G$ embeds in $GL(n/e,q^{e}).e$. Let $W = V \otimes_{F} K$. Now $W
\cong \oplus_{\sigma \in \mathrm{Gal}(K|F)} U^{\sigma}$ as an
$FA$-module. Then $G$ permutes the $U^{\sigma}$. Moreover, $A$
acts irreducibly on $U$ (or $G$ acts reducibly on $W$, a
contradiction to the fact that $V$ is absolutely irreducible as an
$FG$-module). Also $A$ acts faithfully on $U$ ($x$ trivial on $U$
implies that $x$ is trivial on $U^{\sigma}$ for all $\sigma$,
whence $x$ is trivial on $W$). By the maximality of $K$ it follows
that $A$ preserves no field extension structure on $U$. We may
assume that $A$ is absolutely irreducible on $U$. The group $A$
has no non-abelian component for such a subgroup would also be a
component of $G$. Hence the generalized Fitting subgroup of $A$ is
nilpotent.

Suppose that $A$ is not abelian. Then $A$ and $G$ are as in
Sections 4 and 5. Hence Theorem \ref{section5} gives $k(GV) \leq
\{ |V|, 2^{1344} \}$.

Finally, if $A$ is abelian, then $G$ is meta-cyclic. Hence, by
Theorem \ref{main}, we have $k(GV) \leq \max \{ |V|, 5 \}$.

\section{The imprimitive case}

In this section we prove Theorem \ref{irreducible}.

Let $F$ be a finite field, $V$ an $n$-dimensional vector space
over $F$ and also an irreducible $FG$-module for a finite group
$G$. Suppose that $V$ admits a direct sum decomposition $V = V_{1}
\oplus \ldots \oplus V_{t}$ as an $FG$-module. Let $H$ be the
kernel of this linear $G$-action on the $t$ direct summands. In
particular, $H$ is a normal subgroup in $G$ and $G/H$ can be
considered as a permutation group of degree $t$. Suppose that
$V_{1}$ is an irreducible $FL$-module where $L$ is the stabilizer
of $V_{1}$ in $G$. Then $H$ is a subdirect product of irreducible
subgroups $H_{1}, \ldots , H_{t}$ on the vector spaces $V_{1},
\ldots , V_{t}$, respectively. This means that $H$ is such a
subgroup of $H_{1} \times \ldots \times H_{t}$ that projects onto
every direct factor. Suppose that for every normal subgroup $N$ of
any $H_{i}$ we have $k(N) < |V_{i}|/\sqrt{3}$. We will prove by
induction on $t$ that $k(H) < {(|V_{1}|/\sqrt{3})}^{t}$. The case
$t=1$ is trivial. Suppose that $t > 1$ and the statement is true
for $t-1$. Let $H^{*}$ be the projection of $H$ onto all but the
last direct factor of $H_{1} \times \ldots \times H_{t}$. Let the
kernel of this map be $K$. Then $H/K \cong H^{*}$ which in turn is
a subdirect product of $H_{1}, \ldots , H_{t-1}$. By the induction
hypothesis we have $k(H^{*}) < {(|V_{1}|/\sqrt{3})}^{t-1}$. Now $K
\lhd H_{t}$ hence $k(K) < (|V_{t}|/\sqrt{3})$. By Lemma \ref{lN},
we have $k(H) < {(|V_{1}|/\sqrt{3})}^{t}$. By \cite{M} this gives
$$k(G) \leq
k(H) k(G/H) < (|V|/{(\sqrt{3})}^{t}) {(\sqrt{3})}^{t-1} =
|V|/\sqrt{3}$$ for $t > 2$ and $k(G) <(2/3)|V|$ for $t=2$.

This proves Theorem \ref{irreducible}.

\end{document}